\documentclass[12pt,leqno,a4paper]{article}
\usepackage{hyperref}
\usepackage{amsmath}
\usepackage{amssymb}
\usepackage{amsthm}
\usepackage{graphicx,epstopdf}
\usepackage{color}
\hypersetup{
    colorlinks=true,
    linkcolor=black,
    filecolor=black,      
    urlcolor=blue,
    citecolor=red,
}

\makeatletter
\newcommand{\guio}[1]{\nobreakdash-\hspace{0pt}}

\newtheorem*{theorem*}{Theorem}

\newtheorem*{lemmam*}{Lemma 5. (\cite{MOV})}
\newtheorem*{lemma*}{Lemma}
\newtheorem*{corollary*}{Corollary}

\newtheorem{remark*}{Remark}

\theoremstyle{definition}

\newtheorem*{acknowledgements}{Acknowledgements}

\newcommand{\Rn}{\mathbb{R}^n}

\newcommand{\R}{\mathbb{R}}
\newcommand{\C}{\mathbb{C}}

\newcommand{\ep}{\epsilon}

\def\newchi{\raise2pt\hbox{$\chi$}}
\title{The regularity of the boundary of vortex patches for some non-linear transport equations}

\author{J.C.\ Cantero, J.\ Mateu, J.\ Orobitg and J.\ Verdera}

\date{}

\begin{document}
\maketitle
\begin{abstract}
We prove the persistence of boundary smoothness of vortex patches for a non-linear transport equation in $\R^n$ with velocity field given by convolution
of the density with an odd kernel, homogeneous of degree $-(n-1)$ and of class $C^2(\R^n\setminus\{0\}, \R^n).$  This allows the velocity field to have non-trivial divergence. The quasi-geostrophic equation in $\R^3$  and the Cauchy transport equation in the plane are examples.
\bigskip


\medskip

\end{abstract}

\section{Introduction}
The vorticity form of the Euler equation in the plane is
%\begin{equation}\label{vor}
%e
%\end{equation}
\begin{equation}\label{vor}
\begin{split}
&\partial_{t} \omega (x,t)+ v(x,t)\cdot \nabla \omega (x,t)=0,\\
&v(x,t)=(\nabla^{\perp} N*\omega (\cdot, t))(x),\\
&\omega(x,0)=\omega_{0}(x),
\end{split}
\end{equation}
where $x\in\R^2,$ $t\in\R,$ $N=\frac{1}{2\pi}\log|x|$ is the fundamental solution of the laplacian in the plane, $\nabla^{\perp} N$ is a rotation of $\nabla N$ of $90$\textdegree\, in the counterclockwise direction  and $\omega_0$ is the initial vorticity. A deep result of Yudovich \cite{Y} asserts that the vorticity equation is well posed in $L^\infty_c,$  the measurable bounded functions with compact support. A vortex patch is the special weak solution of \eqref{vor}
when the initial condition is the characteristic function of a bounded domain $D_0.$ Since the vorticity equation is a transport equation, vorticity is conserved along trajectories and thus $\omega(x,t)=\chi_{D_t}(x)$ for some domain $D_t.$ A challenging problem,
raised in the eighties, was to show that boundary smoothness persists for all times. Specifically, if $D_0$ has boundary of class $C^{1+\gamma}, \; 0< \gamma<1,$ then one would like $D_t$ to have boundary of the same class for all times. This was viewed as a $2$ dimensional problem which featured some of the main difficulties of the regularity problem for the Euler equation in $\R^3.$  It was conjectured, on the basis of numerical simulations,  that the boundary of $D_t$ could become of infinite length in finite time \cite{M}. Chemin proved that boundary regularity persists for all times \cite{Ch} using paradifferential calculus, and Bertozzi and Constantin found shortly after a minimal beautiful proof in \cite{BC} based on methods of classical analysis with a geometric flavor.

The vortex patch problem was considered for the aggregation equation with newtonian kernel in higher dimensions in \cite{BGLV}. The equation is
\begin{equation}\label{agg}
\begin{split}
&\partial_{t}\rho(x,t)+\operatorname{div}(\rho(x,t) v(x,t) )=0,\\
&v(x,t)=-(\nabla N*\rho(\cdot, t))(x),\\
&\rho (x,0)=\rho_0(x),
\end{split}
\end{equation}
$x\in\Rn$ and $t\in\R.$
In \cite{BLL} a well-posedness theory in $L^\infty_c$ was developed, following the path of \cite{Y} and \cite[Theorem 8.1]{MB}. When the initial condition is the characteristic function of a bounded domain one calls the unique weak solution a vortex patch, as for the vorticity equation. One proves in \cite{BGLV} that if the boundary of $D_0$ is of class $C^{1+\gamma}, \; 0 < \gamma < 1,$ then the solution of \eqref{agg} with initial condition $\rho_0=\chi_{D_0}$ is of the form
$$\rho(x,t) = \frac{1}{1 -t} \newchi_{D_t}(x), \; x \in \Rn, \;
0 \le t < 1,$$
where $D_t$ is a  $C^{1 + \gamma}$ domain for all $t < 1$. The restriction to times less than $1$ obeys a  blow up phenomenon studied in \cite{BLL}. Hence the preceding result is the analog of Chemin's theorem for the aggregation equation. See \cite{BK} for a more general result concerning striated regularity.

After a change in the time scale the aggregation equation for vortex patches becomes the non-linear transport equation
%\begin{equation}\label{agg}
%e
%\end{equation}
\begin{equation}\label{aggtrans}
\begin{split}
&\partial_{t}\rho(x,t)+v(x,t)\cdot \nabla \rho(x,t)=0,\\
&v(x,t)=-(\nabla N*\rho(\cdot, t))(x),\\
&\rho (x,0)=\chi_{D_{0}}(x),
\end{split}
\end{equation}
$x\in\Rn,$ $t\in\R,$
where $N$ is the fundamental solution of the laplacian in $\Rn$ and $D_0$ is a bounded domain.   In this formulation one proves in \cite{BGLV} that if $D_0$ is of class  $C^{1 + \gamma},$  then there is a solution of \eqref{aggtrans} of the form $\chi_{D_t}(x)$ with
$D_t$ a domain of class $C^{1 + \gamma}.$  To the best of our knowledge there is no well-posedness theory in $L^\infty_c$ for \eqref{aggtrans}, for a general initial condition in $L^\infty_c.$  However, if the initial condition is the characteristic function of a domain $D_0$, not necessarily smooth, one has existence and uniqueness for the transport equation \eqref{aggtrans}.  For existence, solve the equation \eqref{agg} with initial condition $\rho_0(x)=\chi_{D_{0}}(x).$ Then the unique solution has the form $\rho(x,t) = \frac{1}{1 -t} \newchi_{D_t}(x)$
and hence, after changing the time scale as in \cite{BLL}, one obtains a solution for \eqref{aggtrans} which is a vortex patch. For uniqueness, we resort to an argument which combines results of \cite{CJM1} and \cite{CJM2} to prove that each weak solution of \eqref{aggtrans} in $L^\infty_c$ is lagrangian and so a vortex patch.
Changing the time scale one obtains a weak solution of \eqref{agg}, which is unique.

The proof follows the scheme of \cite{BC} and overcomes difficulties related
to the fact that the velocity field has a non-zero divergence and to the higher dimensional context. The reader can consult \cite{BGLV} for connections with the existing literature and for references to models leading to various aggregation equations.

This paper originated from an attempt  to deeply understand the role of the kernel that gives the velocity field. For the aggregation equation the kernel is $-\nabla N$ and for the vorticity equation in the plane the kernel is a rotation of $90$  degrees of $\nabla N.$ These are odd kernels, smooth off the origin and homogeneous of degree $-(n-1).$ We wondered what  would happen for the Cauchy kernel 
\begin{equation*}\label{cau}
\frac{1}{2\pi z}= L(\nabla N), \quad\quad \text{with}\quad\quad   L(x,y)=(x,-y), \quad z=(x,y)\in \R^2=\C.
\end{equation*}
 Although apparently there is no model leading to the non-linear transport equation given by the Cauchy kernel, from the mathematical perspective the question makes sense. We then embarked in the study of the non-linear transport equation
%\begin{equation}\label{transcau}
%e
%\end{equation}
\begin{equation}\label{transcau}
\begin{split}
&\partial_{t}\rho(z,t)+v(z,t)\cdot \nabla \rho(z,t)=0,\\
&v(z,t)=\left(\frac{1}{2\pi z} * \rho (\cdot, t)\right)(z),\\
&\rho(z,0)=\chi_{D_{0}}(z),
\end{split}
\end{equation}
where $z=(x,y)$ is the complex variable and $D_0$ is a bounded domain with $C^{1+\gamma}$ boundary, $0<\gamma<1$. A first remark is that apparently there does not exist a well-posedness theory in $L^\infty_c$ for the equation above, but this does not prevent the study of smooth vortex patches, as a particular subclass of  $L^\infty_c$ enjoying a bit of smoothness. 

 To grasp what could be expected we looked at an initial datum which is the characteristic function of the domain enclosed by an ellipse
\begin{equation*}\label{ell}
D_{0}=\left\{ (x,y)\in \mathbb{R}^{2}:  \frac{x^{2}}{a^{2}}+\frac{y^{2}}{b^{2}}<1\right\}.
\end{equation*}
We proved that there exists a weak solution of \eqref{transcau} of the form $\rho(z,t)=\chi_{D_t}(z)$ with $D_t$ the domain enclosed
by an ellipse with semiaxes $a(t)$ and $b(t)$ collapsing to a segment on the horizontal axis as $t \to \infty.$ 

A key remark is that \eqref{transcau} is not rotation invariant. Fix an angle $0 < \theta < \frac{\pi}{2}$ and consider as initial domain the set enclosed by a tilted ellipse
%\begin{equation}\label{tilted}
%e
%\end{equation}
\begin{equation*}\label{tilted}
D_{0}=e^{i\theta}
\left\{ (x,y)\in \mathbb{R}^{2}:  \frac{x^{2}}{a^{2}}+\frac{y^{2}}{b^{2}}<1\right\}.
\end{equation*}
As before we find a weak solution of \eqref{transcau} of the form $\rho(z,t)=\chi_{D_t}(z)$ with $D_t$ the domain enclosed
by an ellipse with semiaxes $a(t)$ and $b(t)$ forming an angle $\theta(t)$ with the horizontal axis. The evolution is different according to whether $0<\theta\le \frac{\pi}{4}$ or $  \frac{\pi}{4} <\theta <  \frac{\pi}{2}.$ Under the assumption that $a_0>b_0,$ in the case $0<\theta\le \frac{\pi}{4}$ the semi-axis $a(t)$ increases as $t \to \infty$ to a positive number $a_\infty,$ $b(t)$ decreases to $0$ and $\theta(t)$ decreases to a positive angle $\theta_\infty.$ Hence $D_t$ collapses into an interval on a line forming a positive angle with the horizontal axis.  If $  \frac{\pi}{4} <\theta <  \frac{\pi}{2},$ then for small times $a(t)$ decreases and $b(t)$ increases, so that the ellipse at time $t$ tends initially to become a circle.  This happens until a critical time is reached after which $a(t)$ increases and $b(t)$ decreases. The angle $\theta(t)$ decreases for all positive times and at some point it becomes $\frac{\pi}{4} ;$ 
after that one falls into the regime of the first case and the domain $D_t$ collapses as $t\to \infty,$ into a segment on a line which forms a positive angle with the horizontal axis. The case $a_0<b_0$ is similar and can be reduced to the previous situation by conjugation (symmetry with respect to the horizontal axis).

Detailed proofs of the results just described can be found in section \ref{ell}.  What they show is that the behavior of vortex patches for the Cauchy transport equation can be much more complicated than for the vorticity or aggregation equations. This is also easily understood if one looks at the divergence of the vector field in \eqref{transcau}. If $\partial$ and $\overline{\partial}$ denote respectively the derivatives with respect to the $z$ and $\bar{z}$ variables, then we get
\begin{equation*}\label{dibar}
2 \,\overline{\partial} v(z,t)=  \rho(z,t)
\end{equation*}
and
\begin{equation*}\label{di}
2\, \partial v(z,t)= -\frac{1}{ \pi} \operatorname{p.v.} \int \frac{1}{(z-w)^2} \rho(w,t)\, dA(w)= \operatorname{B}(\rho(\cdot,t))(z),
\end{equation*}
where $\operatorname{B}$ is the Beurling transform, one of the basic Calder\'on-Zygmund operators in the plane. Here $dA$ is 2 dimensional Lebesgue measure. The divergence of $v$ is given by
\begin{equation*}
\begin{split}
\operatorname{div}v= \Re(2\, \partial v) & = - \operatorname{p.v.} \frac{1}{ \pi} \int \Re \left(\frac{1}{(z-w)^2} \right) \rho(w,t)\, dA(w)
\\*[5pt]
& = 
-  \operatorname{p.v.} \frac{1}{ \pi} \left(\frac{x^2-y^2}{|z|^4} \star \rho(\cdot,t)\right)(z).
\end{split}
\end{equation*}
The last convolution is a Calder\'on-Zygmund operator (a second order Riesz transform) and so it does not map bounded
functions into bounded functions. The most one can say a priori on the divergence of the velocity field is that it is a $BMO$ function in the plane, provided the density $\rho(\cdot,t)$ is a bounded function. It is a well-known fact, already used in \cite{BC} and \cite{Ch}, that if
$D$ is a domain with boundary of class $C^{1+\gamma},$ then an even Calder\'on-Zygmund operator applied to $\chi_{D}$
is a bounded function. Thus we indeed expect $\operatorname{div}v$ to be bounded. Nevertheless, the expression of the divergence of the field in terms of a Calder\'on-Zygmund operator applied to the density is potentially difficult to handle.

We have succeeded in proving that there exists a weak solution of \eqref{transcau} of the form
$\chi_{D_t}$ with $D_t$ a domain with boundary of class $C^{1+\gamma}$ for all times $t \in \R.$ This weak solution is unique in the class of characteristic functions of $C^{1+\gamma}$ domains. 

The Cauchy kernel belongs to a wider class for which the preceding well-posedness theorem holds. We refer to the class of kernels in $\Rn$ which are odd, homogeneous of degree $-(n-1)$ and of class $C^2(\R^n\setminus\{0\}, \R^n).$  Interesting examples of such kernels are those of the form $L(\nabla N),$  where $L$ is a linear mapping from $\R^n$ into itself and $N$ is the fundamental solution of the laplacian in $\R^n.$  They are harmonic off the origin. In particular in $\R^3$ one can take
 $L(x_1,x_2,x_3)= \left(-x_2, x_1, 0 \right).$ The corresponding field is divergence free and the associated equation is the well-known quasi-geostrophic equation. See \cite{GHM} for recent results on rotating vortex patches for the quasi-geostrophic equation. 

Our main result is the following.

\begin{theorem*}\label{teo}
Let $k : \R^n \setminus \{0\} \rightarrow \R^n$ be an odd function, homogeneous of degree $-(n-1)$ and of class  $C^2\left(\R^n\setminus\{0\}, \R^n\right).$ Let $D_{0}$ be a bounded domain with boundary of class~$C^{1+\gamma}$, $0<\gamma<1$. Then the non-linear transport equation
\begin{equation}\label{transL}
\begin{split}
&\partial_{t}\rho (x,t)+v(x,t)\cdot \nabla \rho (x,t)=0,\\
&v(x,t)= (k \star \rho (\cdot, t))(x),\\
&\rho (x,0)= \chi_{D_{0}}(x)
\end{split}
\end{equation}
$x\in\R^n, \; t\in \R,$ has a weak solution of the form
$$
\rho(x,t)=\chi_{D_{t}}(x),\quad x\in\mathbb{R}^{n},\quad t\in \R,
$$
with $D_{t}$ a bounded domain with boundary of class~$C^{1+\gamma}$.

This solution is unique in the class of characteristic functions of domains with boundary of class~$C^{1+\gamma}$.
\end{theorem*}

For the notion of weak solution see \cite[Chapter 8]{MB}.

  A remark on the special case in which the kernel $k$ is divergence free is in order. In this case, in particular for the quasigeostrophic equation, one has well-posedness in $L^\infty_c.$ 
Existence can be proved  following closely the argument in  \cite[Chapter 8]{MB} for the vorticity equation (for the smooth case see \cite{C}).  For uniqueness one resorts to \cite{NPS} whenever the kernel has the special form $L(\nabla N)$
with $L$ a linear map from $\R^n$ into itself.  Indeed, in that work uniqueness in $L^\infty_c$ is proven for the continuity equation in higher dimensions with velocity field given by convolution with $\pm \nabla N.$  The changes needed to take care of the case $L(\nabla N)$ are straightforward.  If $k$ is divergence free and satisfies the general hypothesis stated in the Theorem, then one appeals to \cite{CS}, where uniqueness is proved for lagrangian solutions, and to \cite{CJM1}  and  \cite{CJM2}, in which one shows that a weak solution is lagrangian.

The paper is organized as follows. In the next section we present an outline of the proof,   in which only a few facts are proven. The other sections are devoted to presenting complete proofs of our results. Section \ref{vasin} is devoted to an auxiliary result. In section \ref{def} an appropriate defining function for the patch at time $t$ is constructed. Section \ref{comm} deals with the material derivative of the gradient of the defining function and its expression in terms of differences of commutators. In section \ref{hest} we estimate the differences of commutators in the H\"older norm on the boundary via Whitney's extension theorem. Domains enclosed by ellipses as initial patches for the Cauchy transport equation are studied in section \ref{ell} and the unexpected phenomena that turn up along the vortex patch evolution are described in detail. Finally, there is an Appendix on the existence of principal values of singular integrals in a very special context.

Constants will be denoted by $C,$ mostly without an explicit reference to innocuous parameters, and may be different at different occurrences. If $D$ is a domain with smooth boundary $\sigma=\sigma_{\partial D}$ denotes the surface measure on $\partial D$ and when there is no confusion possible we omit the subscript. The exterior unit normal vector to  $\partial D$ at the point $x$ is denoted by
$\vec{n}(x)=(n_1(x), \dots, n_n(x)),$ without explicit reference to the boundary.

\section{Outline of the proof.}\label{out}
The proof follows the general scheme devised in \cite{BC}. There are serious obstructions caused by the fact that the field is not divergence free and we will explain below how to confront them. The reader will find useful to consult \cite{BC} and \cite{BGLV}.

\subsection{The contour dynamics equation.} Assume that one has a weak solution of \eqref{transL} of the form $\rho(x,t)=\chi_{D_t}(x),$
$D_t$ being a bounded domain of class $C^{1+\gamma}$ for $t$ in some interval $[0,T].$ The field $v(\cdot,t)$ is Lipschitz. This is due to the fact that our kernel has homogeneity $-(n-1)$ and so $\nabla v$ is given by a matrix whose entries are even convolution Calder\'on-Zygmund operators applied to the characteristic function of $D_t$ plus, possibly, a constant multiple of such a characteristic function (coming from a delta function at the origin).  Since $D_t$ has boundary of class $C^{1+\gamma}$ all entries of the matrix $\nabla v$ are functions in $L^\infty(\Rn)$ \cite{BC}. Thus the equation of particle trajectories (the flow mapping)
%\begin{equation}\label{flow}
%flow
%\end{equation}
\begin{equation}\label{flowmap}
\begin{split}
\frac{dX(\alpha,t)}{dt}&=v(X(\alpha,t),t),\\*[5pt]
X(\alpha,0)&=\alpha
\end{split}
\end{equation}
has a unique solution and  $X(\cdot,t)$ is a bilipschitz mapping of $\Rn$ into itself, $0 \le t \le T$. Indeed one has the usual estimate
%\begin{equation}\label{Xbi}
%Xbi
%\end{equation}
\begin{equation}\label{Xbi}
\|\nabla X(\cdot,t)\|_{\infty}\le \exp \int^{t}_{0}\|\nabla v(\cdot,s)\|_{\infty}\,ds.
\end{equation}
%%%%%%%%%%%%%%%%%%%%%%%%%%%%%%%%%%%%%%%%%%%%%%%%%%%%%%%%%%%%%%%%
Since $k$ is homogeneous of degree $-(n-1)$ and smooth off the origin we have
%\begin{equation}\label{campfr}
%campfr
%\end{equation}
\begin{equation}\label{}
k = \partial_1 (x_1 k ) + \partial_2 (x_2 k)+\dots +\partial_n (x_n k), \quad x=(x_1,\dots,x_n) \in \R^n\setminus \{0\}.
\end{equation}
This follows straightforwardly from Euler's theorem on homogeneous functions.

Assume that $\rho(x,t)=\chi_{D_t}(x)$ is a weak solution of the general equation \eqref{transL}. The velocity field is
\begin{equation*}\label{}
\begin{split}
v(\cdot,t)=\chi_{D_t}\star k &= \chi_{D_t}\star \left(\partial_1 (x_1 k)+\dots +\partial_n (x_n k)\right)\\*[5pt] 
&= \partial_1 \chi_{D_t}\star (x_1 k)+ \dots+ \partial_n \chi_{D_t}\star (x_n  k). \\*[5pt]
&= -n_1 d\sigma_{\partial D_t}\star (x_1 k)-\dots-n_n d\sigma_{\partial D_t}\star (x_n k).
\end{split}
\end{equation*}
Thus
\begin{equation}\label{vf}
\begin{split}
v(x,t) &=   -\sum_{j=1}^n   \int_{\partial D_t} (x_j-y_j) k(x-y) n_j(y)\, d\sigma_{\partial D_t}(y), \\*[5pt]
&= -   \int_{\partial D_t} k(x-y) \langle x-y,  \vec{n}(y) \rangle \, d\sigma_{\partial D_t}(y) , \quad x \in \R^n.
\end{split}
\end{equation}
The next step is to set $x=X(\alpha,t)$ and to make the change of variables $y=X(\beta,t)$ in the preceding surface integral. To do this conveniently let $T_1(\beta), ..., T_{n-1}(\beta)$ be an
orthonormal basis of the tangent space to $\partial D_0$ at the
point $\beta \in \partial D_0$ and let $DX(\cdot,t)$ be
the differential of $X(\cdot,t)$ as a differentiable mapping from $\partial D_0$ into $\Rn.$  The vectors $DX (\beta, t) (T_{j}(\beta))$ are tangent to $\partial D_t$ at the point $X(\beta,t)$ for $ 1\le j\le n-1.$ Hence the vector
\begin{equation}\label{extprod}
\bigwedge^{n-1}_{j=1} DX (\beta, t) (T_{j}(\beta))
\end{equation}
is orthogonal to $\partial D_t$ at the point $X(\beta,t)$ and a
different choice of the orthonormal basis $T_j(\beta), 1\leq j\leq
n-1,$  has the effect of introducing a $\pm$ sign in front of
\eqref{extprod}. We may choose the $T_j(\beta)$ so that
$\vec{n}(\beta), T_{1}(\beta), \dots, T_{n-1}(\beta)$ 
gives the standard orientation of $\mathbb{R}^n$.
Substituting the expression \eqref{vf} for the velocity field in \eqref{flowmap} and making the change of variables $y=X(\beta,t)$ we get
\begin{equation*}\label{flow2}
\begin{split}
&\frac{d}{dt}X(\alpha,t)=v(X(\alpha,t),t) \\*[5pt]
& = -   \int_{\partial D_0} k(X(\alpha,t)-X(\beta,t))  \left\langle X(\alpha,t) -X(\beta,t), \bigwedge^{n-1}_{j=1} DX (\beta, t) (T_{j}(\beta)) \right\rangle \, d\sigma_{\partial D_0}(\beta)
\end{split}
\end{equation*}
Let $X: \partial D_0 \rightarrow  \Rn$ be a mapping of class $C^{1+\gamma},$  such that for some constant $\mu >0$
\begin{equation}\label{Xb}
|X(\alpha)-X(\beta)|\ge  \frac{1}{\mu} |\alpha-\beta|,\quad \alpha,\beta \in\partial D_{0}.
\end{equation}
In other words $X\in C^{1+\gamma}(\partial D_0, \Rn), $ $X$ is bilipschitz onto the image and $\mu$ is a Lipschitz constant for the inverse mapping.

Define a mapping
$F(X) : \partial D_0 \rightarrow  \Rn$ by
\begin{equation}\label{efam}
\begin{split}
&F(X)(\alpha) \\*[5pt] 
&=-\int_{\partial D_{0}} k(X(\alpha)-X(\beta)) \,  \left\langle X(\alpha) -X(\beta), \bigwedge^{n-1}_{j=1} DX (\beta) (T_{j}(\beta)) \right\rangle \, d\sigma_{\partial D_0}(\beta).
\end{split}
\end{equation}
The contour dynamics equation (CDE) is 
%\begin{equation}\label{cde}
%cde
%\end{equation}
\begin{equation*}\label{cde}
\begin{split}
\frac{dX(\alpha,t)}{dt}&=F(X(\cdot,t))(\alpha),\quad \alpha\in \partial D_{0},\\*[5pt]
X(\cdot,0)&=I,
\end{split}
\end{equation*}
where $I$ denotes the identity mapping on $\partial D_0.$ 

We conclude that if there exists a weak solution of the type we are looking for, then the flow restricted to $\partial D_0$ is a solution of the CDE.

To proceed in the reverse direction, we need some preparation. Let $\Omega$ be the open set in the Banach space 
$C^{1+\gamma}(\partial D_0, \Rn)$ consisting of those $X \in C^{1+\gamma}(\partial D_0, \Rn)$ satisfying \eqref{Xb} for some $\mu > 0.$
The set $\Omega$ is open in $C^{1+\gamma}(\partial D_0, \Rn)$ and the CDE can be thought of as an ODE in the open set $\Omega.$ We want to show that a solution $X(\cdot,t)$ to the CDE in an interval $(-T,T)$ provides
a weak solution of the non-linear transport equation \eqref{transL}. Clearly $X(\cdot,t)$ maps $\partial D_0$ onto a $n-1$ dimensional hypersurface $S_t.$  The goal now is to identify an open set $D_t$ with boundary $S_t.$  If we add the hypothesis that $\partial D_0$ is  connected, and hence a connected $n-1$ dimensional hypersurface of class $C^{1+\gamma},$ then the analog of the Jordan curve theorem holds \cite[p.89]{GP}. Then the complement of $\partial D_0$ in $\Rn$ has only one bounded connected component which is  $D_0.$  In the same vein, the complement of $S_t$ has only one bounded connected component, which we denote by $D_t,$  so that the boundary of $D_t$ is $S_t.$ The definition of $D_t$ is less direct if we drop the assumption that $\partial D_0$ is connected. We proceed as follows.
Let $S_t^j,  1 \le j \le m,$ be the connected components of $S_t.$  Denote by $U_t^j$  the bounded connected component of  the complement of $S_t^j$ in $\Rn.$  Among the $U_t^j$ there is one, say $U_t^1,$ that contains all the others.  This is so at time $t=0$ because $D_0$ is connected and this property is preserved by the flow $X(\cdot,t).$ We set 
$D_t= U_t^1 \setminus (\cup_{j=2}^m \overline{U}_t^j),$ so that the boundary of $D_t$ is $S_t.$

Indeed, as the reader may have noticed, it is not necessary to assume that $D_0$ is connected in our Theorem. It can be any bounded open set with $C^{1+\gamma}$ boundary. Then the argument we have just described is applied to each connected component.

Define a velocity field by
\begin{equation}\label{velocity}
v(x,t)= \left(k \star \chi_{D_t} \right)(x), \quad x \in \Rn,  \quad t \in (-T,T).
\end{equation}
Since $D_t$ has boundary of class $C^{1+\gamma},$ the field $v(\cdot,t)$ is Lipschitz for each $t\in (-T,T)$ and
the equation of the flow \eqref{flowmap} has a unique solution which is a bilipschitz mapping of $\Rn$ onto itself whose restriction to 
$\partial D_0$ is the solution of the CDE we were given. Thus $X(D_0,t) =D_t $ and $\chi_{D_t}$ is a weak solution of the non-linear transport equation \eqref{transL}.

\subsection{The local theorem.}  As a first step we solve the CDE locally in time. For this we look at the CDE as an ODE in the open set $\Omega$ of the Banach space $C^{1+\gamma}(\partial D_0, \Rn).$ To show local existence and uniqueness we apply the Picard theorem. First one has to check that $F(X) \in C^{1+\gamma}(\partial D_0, \Rn)$ for each $X \in \Omega.$ After taking a derivative in $\alpha$ in \eqref{efam} one gets a $\operatorname{p.v.}$ integral on $\partial D_0,$ 
which defines a Calder\'on-Zygmund operator (not of convolution type) with respect to the underlying measure $d\sigma_{\partial D_0},$  acting on a function satisfying a H\"older condition of order $\gamma.$ The result is again a H\"older function of the same order, since one shows that Calder\'on-Zygmund operators of the type one gets preserve H\"older spaces.  In a second step one needs to prove that
$F(X)$ is locally a Lipschitz function of the variable $X$ or, equivalently, that the differential $DF(X)$ of $F$ at the point $X \in \Omega$  is locally bounded in $X.$  Again one has to estimate operators of Calder\'on-Zygmund  type with respect H\"older spaces of order $\gamma.$ These estimates, subtle at some points, are proved in full detail  in \cite{BGLV} for the kernel $k=-\nabla N.$  The variations needed to cover the present situation are minor and are left to the reader. It is important that, as in \cite{BGLV}, the time interval on which the local solution exists depends continuously only on
the dimension $n,$ the kernel $k$, the diameter of $D_0,$ the $n-1$ dimensional surface measure of $\partial D_0$ and the constant  $q(D_0)$ determining the $C^{1+\gamma}$ character of $\partial D_0,$  whose definition we discuss below. 

Let $D$ be a bounded domain with boundary of class $C^{1+\gamma}.$ Then there exists a defining function of class $C^{1+\gamma}$, that is, a function $\varphi \in C^{1+\gamma}(\Rn),$ such that $D=\{x \in \Rn : 	\varphi(x) < 0 \}$ and $\nabla \varphi(x)\neq 0$ if $\varphi(x)=0.$  We set
\begin{equation}\label{kiu}
q(D)= \inf \{\frac{\| \nabla \varphi\|_{\gamma, \partial D}}{|\nabla \varphi|_{\inf}} :  \varphi \; \text{a defining function of $D$ of class}\; C^{1+\gamma} \},
\end{equation}
where $|\nabla \varphi(x)| = \sqrt {\sum_{j=1}^n \partial_j \varphi(x)^2},$
\begin{equation*}
\begin{split}
\| \nabla \varphi\|_{\gamma, \partial D} &= \sup \{ \frac{|\nabla \varphi(x) -\nabla \varphi(y)|}{|x-y|^\gamma}: x, y \in \partial D, x\neq y \},\\
|\nabla \varphi|_{\inf}&= \inf \{|\nabla \varphi(x)| : \varphi(x)=0\}.
\end{split}
\end{equation*}

There is here an important variation with respect to \cite{BC} and  \cite{BGLV}: the H\"older seminorm of order $\gamma$  of $\nabla \varphi$ is taken  in those papers in the whole of $\Rn.$ For reasons that will become clear later on we need to restrict our attention to the boundary of $D$ and this requires finer estimates.

\subsection{Global existence: a priori estimates.}  Assume that the maximal time of existence for the solution $X(\cdot,t)$ of the CDE is
$T.$  By this we mean that $X(\cdot,t)$ is defined for $t \in (-T,T)$ but cannot be extended to a larger interval. We want to prove that $T=\infty.$ For that it suffices to prove that for some constant $C=C(T)$ one has
\begin{equation}\label{apest}
\operatorname{diam}(D_t) + \sigma(\partial D_t)+ q(D_t) \le C, \quad t \in (-T,T).
\end{equation}
If the preceding inequality holds, then we take $t_0 < T$ close enough to $T$ so that after the application of the existence and uniqueness theorem for the CDE to the domain $D_{t_0}$ at time $t_0$ we get an interval of existence for the solution which goes beyond $T$ (the same argument applies to the lower extreme $-T$).

To obtain \eqref{apest} we look for a priori estimates in terms of $\|\nabla v \|_\infty.$  For $\operatorname{diam}(D_t)$ and  $\operatorname{\sigma}(\partial D_t)$ this is straightforward in view of \eqref{Xbi}. The core of the paper is the
a priori estimate of $q(D_t),$  which we get by constructing an appropriate defining  function $\Phi(\cdot,t)$ for $D_t$ satisfying 
\begin{align}
|\nabla \Phi(\cdot,t)|_{\inf} \ge |\nabla \varphi_0|_{\inf} \, \exp \left(-C_n\, \int_0^t \|\nabla v(\cdot,s)\|_{\infty} \, ds \right), \quad t>0,\label{ap1}\\
\|\nabla \Phi(\cdot,t)\|_{\gamma, \partial D_t} \le \|\nabla
\varphi_0 \|_{\gamma, \partial D_0} \,\exp \left(C_n\, \int_0^t (1+\|\nabla v(\cdot,s)\|_\infty)\,ds   \right), \quad t>0.\label{ap2}
\end{align}

As it was pointed out in \cite{BGLV} if one transports a defining function $\varphi_0$ of $D_0$ by $\varphi_t= \varphi_0 \circ X^{-1}(\cdot,t),$ then $\nabla \varphi_t$ may have jumps at the boundary of $D_t$  for $t\neq 0$ and so  $\varphi_t$ is not necessarily differentiable. In \cite{BGLV} one shows that
%\begin{equation}\label{limdins}
%limdins
%\end{equation}
\begin{align}\label{limdins}
\lim_{D_{t}\ni y\to x}\nabla \varphi (y,t)= \lim_{D_{t}\ni y\to x}\det \nabla X^{-1} (y,t) \, \frac{|\nabla \varphi_{0}(X^{-1}(x,t))|}{\det D(x)}\,\vec{n}(x),\quad x\in \partial D_{t} \\
\lim_{\mathbb{R}^{n} \setminus \overline{D}_{t}
\ni y\to x}\nabla \varphi (y,t)= \lim_{\mathbb{R}^{n} \setminus  \overline{D}_{t}
\ni y\to x}\det  \nabla X^{-1} (y,t) \, \frac{|\nabla \varphi_{0}(X^{-1}(x,t))|}{\det D(x)}\,\vec{n}(x), 
\quad x\in \partial D_{t},\label{limfora}
\end{align}
where $X^{-1}(\cdot,t)$ is the inverse mapping of $X(\cdot,t)$ and $D(x)$ is the differential at $x$ of the restriction of $X^{-1}(\cdot,t)$ to $\partial D_t,$ as a differentiable mapping from $\partial D_t$ onto $\partial D_0.$ Define
%\begin{equation}\label{deffun}
%deffun
%\end{equation}
\begin{equation}\label{deffun}
\Phi (x,t)=\begin{cases}
0, &x \in \partial D_{t},\\
\det \nabla X \big(X^{-1}(x,t),t \big)\;\varphi(x,t), &x\notin \partial D_{t}.
\end{cases}
\end{equation}
We show in section \ref{def} that $\Phi(x,t)$ is a defining function of $D_t$ of class $C^{1+\gamma}.$

The definition of $\Phi$ yields a formula for its material derivative $\frac{D}{Dt}=\partial_t + v \cdot \nabla,$  namely,
\begin{equation}\label{matd}
\frac{D \Phi}{Dt}=\operatorname{div} (v) \, \Phi.
\end{equation}
Taking gradient in the preceding identity one gets 
\begin{equation}\label{matdgrad1}
\frac{D (\nabla \Phi)}{Dt}= \nabla \big(\operatorname{div} (v)\big) \, \Phi+ \operatorname{div}(v) \nabla \Phi - (\nabla v)^t (\nabla \Phi),
\end{equation}
where $(\nabla v)^t$ stands for the transpose of the matrix $\nabla v.$  The right hand side of \eqref{matdgrad1} can be split into two terms which behave differently. The first is $ \nabla \big(\operatorname{div} (v)\big) \, \Phi$ and the second  $\operatorname{div}(v) \nabla \Phi - (\nabla v)^t (\nabla \Phi).$   We prove that the second term is a finite sum of differences of commutators, which can be shown, with some effort, to have the right estimates. The first term does not combine with others to yield a commutator and because of that we call it the solitary term. A priori it is the most singular term on the right hand side of \eqref{matdgrad1}, since it contains second order derivatives of $v.$  We show that  the solitary term extends continuously to $\partial D_t$ by $0$ and so it can be ignored at the price of working only on the boundary of $D_t$ for all $t.$ 

To prove that the solitary term extends continuously to the boundary by $0$ we need a recent result of Vasin \cite{V} whose statement is as follows. Let $T$ be a convolution homogeneous even Calder\'on-Zygmund operator of the type
\begin{equation}\label{cz}
T(f)(x) = \operatorname{p.v.} \int_{\R^n} L(x-y) f(y) \,dy = \lim_{\ep \to 0} \int _{|y-x|>\ep} L(x-y) f(y) \,dy, 
\end{equation}
where $L$ is an even kernel, homogeneous of degree $-n,$  satisfying the smoothness condition $L \in C^1(\R^n\setminus \{0\})$ and the cancellation property $\int_{|x|=1} L(x) \,d\sigma(x)=0.$  The function $f$ is in $L^p(\R^n), \; 1\le p< \infty$ and the principal value integral \eqref{cz} is defined a.e. on $\R^n.$ Vasin's result states that if $D$ is a bounded domain with boundary of class $C^{1+\gamma}$ then
\begin{equation}\label{va}
\left|\nabla T(\chi_D)(x)\right| \, \operatorname{dist}(x,\partial D)^{1-\gamma} \le C,  \quad x \in D\cup \left(\R^n\setminus \overline{D}\right),
\end{equation}
where the constant $C$ depends only on $n, \gamma$ and the constants giving the smoothness of $\partial D.$ We provide a proof of \eqref{va} in section \ref{vasin} for completeness.

One applies \eqref{va} to the second derivatives of the velocity field $v=k\star \chi_{D_t}$ with $t$ fixed. One has in the distributions sense
\begin{equation}\label{derker}
\partial_j k =   \operatorname{p.v.} \partial_j k+ \vec{c_j} \, \delta_0, 	\quad \text{with}\quad \vec{c_j} =\int_{|\xi|=1} k(\xi) \xi_j \, d\sigma(\xi),
\end{equation}
and so
\begin{equation*}\label{}
(\partial_j v) (x)=   (\operatorname{p.v.} \partial_j k \star \chi_{D_{t}})(x)+ \vec{c_j}  \,\chi_{D_{t}}(x),  \quad x \in D_t\cup \left(\R^n\setminus \overline{D_t}\right)
\end{equation*}
and, taking a second derivative,
\begin{equation}\label{secder}
(\partial_l\partial_j v) (x)=  \partial_l  (\operatorname{p.v.} \partial_j k \star \chi_{D_{t}})(x),  \quad x \in D_t\cup \left(\R^n\setminus \overline{D_t}\right), \quad 1\le l, j \le n.
\end{equation}
By \eqref{va} applied to the operator $T$ associated with the kernel $L=\partial_j k $ 
\begin{equation}\label{disbou}
\left| (\partial_l \partial_j v) (x)\right|  	\operatorname{dist}(x,\partial D_t)^{1-\gamma} \le C(t),  \quad x \in D_t\cup \left(\R^n\setminus \overline{D_t}\right), \quad  1\le l, j \le n,
\end{equation}
where $C(t)$ depends on $n, \gamma$ and the constants related to the smoothness of $\partial D_t.$ This implies that the solitary term  has limit $0$ at the boundary of $\partial D_t,$ coming from the complement, because $|\Phi(x,t)|$ is comparable to $\operatorname{dist}(x,\partial D_t)$ as $x$ approaches $\partial D_t $ ($\Phi(\cdot,t)$ is continuously differentiable and vanishes on the boundary but the gradient does not).

It is worth remarking that if each component of the kernel $k$ is harmonic off the origin, then \eqref{va} can be obtained readily from the fact that 
$T(\chi_D)$ satisfies a H\"older condition of order $\gamma$ in $D,$  which is the Main Lemma of \cite{MOV}.

From \eqref{matdgrad1} at boundary points, and thus without the solitary term, one gets straightforwardly \eqref{ap1}. 
Thus the a priori estimate of $q(D_t)$ is reduced to \eqref{ap2}.

We turn now our attention to the second term in \eqref{matdgrad1}. 
We prove that the $i$-th component of the vector $\operatorname{div}(v) \nabla \Phi - (\nabla v)^t (\nabla \Phi)$ evaluated at the point $x \in \R^n$ is a sum of $n-1$ terms, each of which is a difference of two commutators. In fact, the $i$-th component is
\begin{equation}\label{commind}
\sum_{j\neq i}\operatorname{p.v.} \int_{D_t} \partial_j k_j(x-y) \left(\partial_i \Phi(x)-  \partial_i \Phi(y)\right)\, dy -
\operatorname{p.v.} \int_{D_t} \partial_i k_j(x-y) \left(\partial_j \Phi(x)-  \partial_j \Phi(y)\right)\, dy.
\end{equation}
 It is crucial here that we obtain differences of commutators, which provides eventually an extra cancellation.

In \cite{BC}  it was shown that the H\"older semi-norm of order $\gamma$ of each commutator in \eqref{commind} can be estimated by $C_n \,(1+\|\nabla v(\cdot,t)\|_\infty) \|\nabla \Phi(\cdot,t)\|_{\gamma, \R^n}.$   This is not enough in our situation, because of the presence of the factor $\|\nabla \Phi(\cdot,t)\|_{\gamma, \R^n},$ which should be replaced by a boundary quantity like $\|\nabla \Phi(\cdot,t)\|_{\gamma, \partial D_t}.$

 To obtain the correct estimate we transform the $j$-the term in \eqref{commind} into a difference of two boundary commutators:
\begin{equation}\label{commfr}
\begin{split}
\operatorname{p.v.} \int_{\partial D_t}  k_j(x-y) & \left(\partial_i \Phi(x) - \partial_i \Phi(y)\right)  n_j(y)\, d\sigma(y)  \\*[5pt] 
&-\operatorname{p.v.} \int_{\partial D_t} k_j(x-y) \left(\partial_j \Phi(x)- \partial_j \Phi(y)\right) n_i(y)\, d\sigma(y).
\end{split}
\end{equation}
It is worth emphasizing here that it is not true that the commutator
\begin{equation*}\label{}
\operatorname{p.v.} \int_{D_t} \partial_j k_j(x-y) \left(\partial_i \Phi(x)-  \partial_i \Phi(y)\right)\, dy
\end{equation*}
equals
\begin{equation*}\label{}
\operatorname{p.v.} \int_{\partial D_t}  k_j(x-y)  \left(\partial_i \Phi(x) - \partial_i \Phi(y)\right)  n_j(y) \, d\sigma(y).
\end{equation*}
What is true is that the difference of two commutators in the $j$-th term of  \eqref{commind} equals the difference of two commutators in  \eqref{commfr}.  There is some magic here in arranging all terms so that certain hidden cancellation takes place. To get the right estimates on the boundary commutators one cannot adapt \cite[Lemma 7.3, p.26]{BC} to the underlying measure $ d \sigma$ on $\partial D_t,$ because this would give a constant of the type
\begin{equation*}\label{growt}
  C_t = \sup_{x\in\partial D_t }\sup_{r>0} \frac{\sigma (B(x,r))}{ r^{n-1}},
\end{equation*}
which can be estimated by the Lipschitz constant of $X(\cdot,t),$ namely, $\exp \int_0^t \|\nabla v(\cdot,s)\|_\infty \,ds.$  This exponential constant is by far too large.

One needs to replace the standard bound  $C_n\, (1+\|\nabla v(\cdot,t)\|_\infty \|) \nabla \Phi(\cdot,t)\|_{\gamma, \R^n}$ for a ``solid" commutator of the type \eqref{commind} by $C_n \,(1+\|\nabla v(\cdot,t)\|_\infty ) \|\nabla \Phi(\cdot,t)\|_{\gamma, \partial D_t}.$  Here we have used the term solid commutator to indicate that the integration is on $D_t$ with respect to $n$-dimensional Lebesgue measure as opposed to a boundary commutator in which the integration is on the boundary of $\partial D_t$ with respect to surface measure $\sigma.$
To get the estimate in terms of $\|\nabla \Phi(\cdot,t)\|_{\gamma, \partial D_t}$  we resort to the difference of commutators structure,  which allows us to appeal to Whitney's extension theorem, the reason being that one can switch between a difference of boundary commutators and a difference of solid commutators via the divergence theorem. The final outcome is \eqref{ap2}.

Of course for those cases in which the kernel is divergence free, the quasi-geostrophic equation in particular, one does not need the boundary commutators and getting the commutator formula \eqref{commind} suffices to complete the proof as in \cite{BC}. Indeed in these cases the transported defining function is already a genuine defining function, since the gradient has no jump according \eqref{limdins} and \eqref{limfora}, or appealing to a regularization argument, as in \cite{R}.

To complete the proof from the a priori estimates  is a standard reasoning. One needs a logarithmic inequality for $\|\nabla
v(\cdot,t)\|_\infty,$ which is a consequence of the boundedness of $T(\chi_{D})$ for an even smooth convolution Calder\'on-Zygmund operator $T$ and a domain $D$ with boundary of class $C^{1+\gamma},$ and of the particular form of the constant. One obtains

\begin{equation}\label{login}
\begin{split}
||\nabla v(\cdot,t)||_{\infty} & \leq \frac{C_n}  {\gamma}  \left(1 + \log^+
\left(|D_t|^{1/n} \frac{||\nabla \Phi||_{\gamma, \partial D_t}}  {|\nabla \Phi|_{\rm
{inf}}} \right)\right),
%\\*[5pt]& = \frac{C_n}  {\gamma}  \Big(1 + \log^+\left(|D_t|^{1/n} q(D_t) \right)\Big) 
\end{split}
\end{equation}
where $C_n$ is a dimensional constant and $|D|$ stands for the $n$-dimensional Lebesgue measure of the measurable set $D.$  The novelty in inequality \eqref{login} is that $||\nabla \Phi||_{\gamma, \partial D_t}$  is now replacing the larger constant $||\nabla \Phi||_{\gamma, \R^n}$ which appears in \cite{BC} or \cite[Corollary 6.3]{BGLV} in dealing with the corresponding inequality. This follows from a scrutiny of the constants that appear along the proof and an application of the implicit differentiation formula.

Inserting \eqref{ap1} and \eqref{ap2}  in \eqref{login} one gets, for a dimensional constant $C,$
\begin{equation*}\label{gradvuniform}
||\nabla v (\cdot,t)||_{\infty} \le C + C \int_0^t (1+||\nabla v
(\cdot,s)||_{\infty}) \, ds,
\end{equation*}
which yields, by Gronwall,
\begin{equation*}\label{expgradv}
||\nabla v (x,t)||_{\infty} \le C \,e^{C t}, \quad -T< t< T,
\end{equation*}
and this completes the proof of \eqref{apest}.

The reader may have observed that it is not strictly necessary for the proof to use the quantity $q(D_t),$ defined in \eqref{kiu}. Nevertheless, it is the canonical quantity to take into consideration and helps to make some statements clearer. We will use it again in section \ref{def}.

\section{An auxiliary result}\label{vasin}
The result we are referring to is the following and can be found in \cite{V}.

\begin{lemma*}\label{vasin}
Let $D 	\subset \mathbb{R}^n$ be a bounded domain with boundary of class $C^{1+\gamma}, \; 0<\gamma<1,$ and $L$ an even kernel in $C^1(\R^n \setminus {0}), $ homogeneous of degree $-n.$ Then
\begin{equation*}\label{vasin2}
|\nabla \left(L\star \chi_D\right) (x) | \operatorname{dist}(x,\partial D)^{1-\gamma}  \le C, \quad x \in \R^n \setminus \partial D,
\end{equation*}
where $C$ is a constant depending only on $D.$
\end{lemma*}

\begin{proof}
Placing the gradient on the characteristic function of $D$ we obtain
\begin{equation*}\label{}
\nabla \left(L\star \chi_D\right) = L\star (-\vec{n} \,d\sigma_{\partial D}).
\end{equation*}
Fix $x\in D$ and set $d=d(x)=\text{dist}(x,\partial D).$ By the divergence theorem
\begin{equation*}\label{}
(L\star \vec{n} \,d\sigma_{\partial D})(x) = (L\star \vec{n} \,d\sigma_{\partial B(x,d)})(x)-\int_{D\setminus B(x,d)} \nabla L(x-y) \,dy,
\end{equation*}
Now
\begin{equation*}\label{}
(L\star \vec{n} \,d\sigma_{\partial B(x,d)})(x)=\int_{|y-x|=d} L(x-y) \vec{n}(y) \,d\sigma(y)=\int_{|z|=d} L(z) \vec{n}(z) \,d\sigma(z)
\end{equation*}
and the last integral clearly vanishes, owing to the oddness of $L(z) \vec{n}(z).$ Thus
\begin{equation*}\label{}
\nabla \left(L\star \chi_D\right)(x)= -(L\star \vec{n} \,d\sigma_{\partial D})(x) = \int_{D\setminus B(x,d)} \nabla L(x-y) \,dy,
\end{equation*}
and
\begin{equation*}\label{}
\begin{split}
\text{dist}(x,\partial D)\, |\nabla \left(L\star \chi_D\right) (x) | & = d\, \left|\int_{D\setminus B(x,d)} \nabla L(x-y) \,dy \right| \\*[5pt]
& \le d \int_{D\setminus B(x,d)} \frac{C}{|y-x|^{n+1}} \,dy   \le C.
\end{split}
\end{equation*}
Therefore in proving the lemma one can assume that $d \le \frac{1}{2} r_0,$  where $r_0=r_0(D)$ has the property that, given a point $p$ in the boundary of $D,$ 
$B(p, 2r_0) \cap D$ is the set of points in $B(p,2r_0)$ lying below the graph of a $C^{1+\gamma}$ function defined on the tangent hyperplane through $p.$

We assume, without loss of generality, that $0$ is the closest point of $\partial D$ to $x$ and that the tangent hyperplane to $\partial D$ at $0$ is $\{x 	\in \R^n: x_n=0\}.$ We also assume that $D\cap B(0,2r_0)=\{x\in \R^n : x_n < \varphi(x') \},$ where $x'=(x_1, \dots, x_{n-1}),$
$\varphi \in C^{1+\gamma}(B'(0,2r_0)),$ $ B'(0,2r_0) =\{x'\in \R^{n-1}: |x'|< 2r_0\}.$ In particular,
\begin{equation*}\label{}
|\varphi(x')| \le \|\nabla \varphi \|_{\gamma, B'(0,2r_0)} |x'|^{1+\gamma}, \quad x' \in B'(0,2r_0).
\end{equation*}

We clearly have

\begin{equation*}\label{}
\int_{D\setminus B(x,d)} \nabla L(x-y) \,dy  = \int_{(D\setminus B(x,d))\cap B(0,r_0)} \nabla L(x-y) \,dy + \int_{D\cap B^c(0,r_0)} \nabla L(x-y) \,dy.
\end{equation*}
The second term above is easy to estimate:
\begin{equation*}\label{}
\left|\int_{D\cap B^c(0,r_0)} \nabla L(x-y) \,dy \right| \le \int_{B^c(0,r_0)} \frac{dy}{|y|^{n+1}} \,dy \le  \frac{C}{r_0}.
\end{equation*}
For the first term one uses the fact that if $H$ is a halfspace then $L\star \chi_{H}$ vanishes on $H.$ This follows from the fact that the preceding statement is true for balls instead of halfspaces \cite{MOV} and a straightforward limiting argument. Then one has
\begin{equation*}\label{}
\begin{split}
\int_{(D\setminus B(x,d))\cap B(0,r_0)} \nabla L(x-y) \,dy &= \int_{(D\setminus B(x,d))\cap B(0,r_0)} \nabla L(x-y) \,dy -\int_{H_{-}} \nabla L(x-y) \,dy \\*[5pt]
& = \int_{(D\setminus H_{-}) \cap B(0,r_0)} \nabla L(x-y) \,dy \\*[5pt] & -\int_{(H_{-}\setminus (D \cup B(x,d)))\cap B(0,r_0)} \nabla L(x-y) \,dy \\*[5pt]
& - \int_{H_{-}\cap B^c(0,r_0)} \nabla L(x-y) \,dy 
\end{split}
\end{equation*}
and the last term estimated as we did above with $D$ in place of $H_{-}.$ The remaining two terms are tangential and they are treated similarly. For the first we set
\begin{equation*}\label{}
\begin{split}
 \int_{(D\setminus H_{-}) \cap B(0,r_0)} \nabla L(x-y) \,dy &=  \int_{(D\setminus H_{-}) \cap B(0,2d)} \nabla L(x-y) \,dy \\*[5pt]
 & + \int_{(D\setminus H_{-}) \cap \left(B(0,r_0)\setminus B(0,2d)\right)} \nabla L(x-y) \,dy.
 \end{split}
\end{equation*}
Since for $x \in D \setminus H_{-}$ one has $|y-x| \ge d,$ we get
\begin{equation*}\label{}
\begin{split}
 \left|\int_{(D\setminus H_{-}) \cap B(0,2d)} \nabla L(x-y) \,dy \right| & \le \frac{C}{d^{n+1}} \left|  (D\setminus H_{-}) \cap B(0,2d) \right|   \\*[5pt]
 & \le  \frac{C}{d^{n+1}}  \int_0^{2d} \rho^{n-1}\, \sigma\{\theta \in S^{n-1} : \rho \theta \in D\setminus H_{-} \} \,d\rho\\*[5pt]
 & \le  \frac{C}{d^{n+1}}  \int_0^{2d} \rho^{n-1+\gamma} \ \,d\rho = C \,d^{\gamma -1}.
 \end{split}
\end{equation*}
Finally
\begin{equation*}\label{}
\begin{split}
 \left|\int_{(D\setminus H_{-}) \cap (B(0,r_0)\setminus B(0,2d))} \nabla L(x-y) \,dy \right| & \le C\, \int_{(D\setminus H_{-}) \cap (B(0,r_0)\setminus B(0,2d))} \frac{1}{|y|^{n+1}} \,dy \\*[5pt]
 & \le  C\, \int_0^{2d} \frac{1}{\rho^{n+1}} \rho^{n-1+\gamma}  \,d\rho  = C \,d^{\gamma -1}.
 \end{split}
\end{equation*}
\end{proof}

It is an interesting fact that the preceding lemma implies the main lemma in \cite{MOV}, which states that under the hypothesis of Vasin's lemma the function $L\star \chi_{D}$ satisfies a H\"older condition of order $\gamma$ on $D$ and on $\R^n \setminus \overline{D}.$
Incidentally, it is worth mentioning that this result has been proved independently by various authors at different times and with various degrees of generality. We are grateful to M. Lanza de Cristoforis for bringing to our attention the oldest reference we are aware of, namely, the 1965 paper of Carlo Miranda \cite{Mi}.

We give an account of the proof of this fact only for the statement concerning $D.$  In the exterior of $D$ one applies similar arguments.

Take two points $x$ and $y$ in $D.$ Let $d=\operatorname{dist}(x,\partial D)$ be the distance from $x$ to the boundary.  As before, we assume, without loss of generality, that $0$ is the closest point of $\partial D$ to $x$ and that the tangent hyperplane to $\partial D$ at $0$ is $\{x \in \R^n: x_n=0\}.$   We can also assume that $D\cap B(0,2r_0)=\{x\in \R^n :   x_n <\varphi(x')  \},$ where $x'=(x_1, \dots, x_{n-1}),$
$\varphi \in C^{1+\gamma}(B'(0,2r_0)),$ $ B'(0,2r_0) =\{x'\in \R^{n-1}: |x'|< 2r_0\}.$  Then
\begin{equation}\label{graf}
|\varphi(x')| \le \|\nabla \varphi \|_{\gamma, B'(0,2r_0)} |x'|^{1+\gamma}, \quad x' \in B'(0,2r_0).
\end{equation}

As in \cite{MOV} we can reduce matters to the case in which $d \le \frac{1}{2} r_0,$  because otherwise we resort to the smoothness of $L\star \chi_{D}$ on the domain $\{z \in D: \operatorname{dist}(z,\partial D) >  \frac{1}{2} r_0\}.$ 

Let $K$ be the closed cone with aperture $45^{\circ}$ and axis the negative $x_n$ axis. That is
$$K=\{x\in \R^n: - \sqrt{2 } \,x_n \ge |x|\}.$$ We say that $x$ and $y$ are in non-tangential position if $x,y \in K.$ Otherwise they are in tangential position.

 Assume first that $x,y \in D$ are in non-tangential position and distinguish two cases. The first is $y \in B(0,2d)\setminus B(0,d).$ Apply the mean value theorem on an arc contained in $K\cap\big(B(0,2d)\setminus B(0,d)\big)$ of length comparable to $|y-x|.$ One gets
\begin{equation}\label{inc}
|f(y)-f(x)| \le C\, \sup \{\operatorname{dist}(\xi,\partial D)^{\gamma-1}: \xi \in K\cap \big(B(0,2d)\setminus B(0,d)\big) \} |y-x|.
\end{equation}
We claim that there exists an absolute constant $c_0$ with $0<c_0 < 1$ satisfying
\begin{equation}\label{dis}
\operatorname{dist}(\xi,\partial D) \ge c_0\, |\xi_n|, \quad \xi \in K \cap B(0,r_0),
\end{equation}
 provided $r_0$ is small enough.
%To show this one first remarks that, taking $r_0$ sufficiently small, the exterior unit normal vector $\vec{n}(p)$ to $\partial D$ at points $p$ of $B(0,2r_0)\cap \partial D$ satisfies
%$|\vec{n}(p) - \vec{n}(0)| <1.$  One also notes that if is a point of $\partial D$ at which the distance from $\xi$ to $\partial D$ is attained, then $\xi_0-\xi$ is normal to $\partial D$ at $\xi_0.$ 
Let $p \in \partial D$ be such that $|\xi-p|=\operatorname{dist}(\xi,\partial D).$ Since $p=(p',p_n)$ is on the graph of $\varphi$ we have, by \eqref{graf}, $|p_n| \le C\, |p'|^{1+\ep} \le C\,r_0^\ep\, |p'|.$ Thus
\begin{equation*}\label{}
\begin{split}
|\xi_n| & \le |\xi_n -p_n| + |p_n| \le  |\xi -p| +C\,r_0^\ep\, |p'|  \\*[5pt]
& \le  |\xi -p| +C\,r_0^\ep\, \big(|p'-\xi'| +|\xi'|\big) \\*[5pt]
&\le |\xi -p| \big(1+C\,r_0^\ep \big)+C\,r_0^\ep\, |\xi_n|,
\end{split}
\end{equation*}
where in the last inequality we used that $|\xi| \le \sqrt{2}|\xi_n|, \; \xi \in K.$ Taking $r_0$ so small that $C\,r_0^\ep \le 1/2$ we obtain
$$|\xi_n|  \le 2 \big( 1+C\,r_0^\ep\big)|\xi -p| =  2 \big( 1+C\,r_0^\ep\big) \operatorname{dist}(\xi,\partial D) .$$
Indeed, the constant $C$ is the previous string of inequalities is $\sqrt{2}\, \|\nabla \varphi \|_{\gamma, B'(0,2r_0)},$ which also depends on $r_0.$ But this is not an obstruction because it decreases with $r_0.$

Therefore, by \eqref{inc},
\begin{equation*}\label{}
\begin{split}
|f(y)-f(x)| & \le C\, (c_0 |\xi_n |)^{\gamma-1}  |y-x| \le C\,c_0^{\gamma-1}  d^{\gamma-1} |y-x|^{1-\gamma}  |y-x|^\gamma \\*[5pt]
& \le C\,c_0^{\gamma-1}  d^{\gamma-1} (3d)^{1-\gamma}  |y-x|^\gamma = C\, |y-x|^\gamma.
\end{split}
\end{equation*}

Let us turn our attention to the case $y \in K\cap B^c(0,2d).$ Note that there exists an absolute constant $C_0 >1$ such that 
$$|y-x| \le C_0\, |y_n-x_n|, \quad y \in K\cap B^c(0,2d).$$
Apply the fundamental theorem of Calculus on the interval with end points $x$ and $y$ and estimate the gradient of $f$ by a constant times the distance to the boundary raised to the power $\gamma-1.$  By \eqref{dis} we obtain
\begin{equation*}\label{}
\begin{split}
|f(y)-f(x)| & \le C\,  \int_0^1 \operatorname{dist}(x+t(y-x),\partial D)^{\gamma-1}  |y-x| \, dt  \\*[5pt] 
& \le C\,c_0^{\gamma-1}  \int_0^1 |x_n+t(y_n-x_n)|^{\gamma-1}  |y-x| \, dt \\*[5pt]
& = C\,  \frac{|y-x|}{|y_n-x_n|} \int_0^{|y_n-x_n|} (d+\tau)^{\gamma-1}  \, d\tau \\*[5pt] 
& = C\,C_0 \, \big((d+|y_n-x_n|)^\gamma -d^\gamma \big) \\*[5pt] 
& \le C\, |y_n-x_n|^\gamma \le C\, |y-x|^\gamma,
\end{split}
\end{equation*}
as desired.

We are left with the case in which $x$ and $y$ are in tangential position, that is, $y \in D \cap \big(\R^n \setminus K\big).$ In \cite{MOV}  there is a reduction argument to the non-tangential case, which we now reproduce for completeness. Take a point $p \in \partial D$ with 
$|y-p|=\operatorname{dist}(y,\partial D)$ and let $\vec{N}$ be the exterior unit normal vector to $\partial D$ at $p.$ We will take $r_0$ so small that $\vec{N}$ is very close to the  exterior unit normal vector $\vec{n}$  to $\partial D$ at $0.$ Then the ray
$y-t \vec{N}, \; t>0,$ will intersect $K$ at some point $y_0$ and the pairs $x, y_0$ and $y,y_0$ will be in tangential position. Let us seek a condition on $t$ so that $y-t \vec{N} \in K,$  that is, so that 
\begin{equation}\label{kai}
|y-t \vec{N}| \le \sqrt{2} |\langle y-t \vec{N}, \vec{n} \rangle|.
\end{equation}
Here $\langle \cdot,\cdot \rangle$ denotes the scalar product in $\R^n.$  Since $|\langle y-t \vec{N}, \vec{n} \rangle| \ge t \langle \vec{N}, \vec{n} \rangle -|y|$ and $|y-t \vec{N}| \le |y|+t,$ a sufficient condition for \eqref{kai} is
\begin{equation*}\label{}
(1+\sqrt{2}) |y| \le t \big( \sqrt{2} \langle \vec{N}, \vec{n} \rangle -1\big).
\end{equation*}
Take $r_0$ small enough so that $\sqrt{2} \langle \vec{N}, \vec{n} \rangle -1 \ge (\sqrt{2}-1)/2.$ A simpler sufficient condition for \eqref{kai} is 
$$ |y| \le c_0 \,t, \quad\quad \text{with} \quad\quad c_0= \frac{1}{2}\,\frac{\sqrt{2}-1}{\sqrt{2}+1}.$$
Define $t_0$ by $|y| = c_0 \,t_0$  and then set $y_0=y-t_0 \vec{N}.$ By construction, $y_0 \in K \cap D$ and the pairs $x, y_0$ and
$y, y_0$ are in non-tangential position. Hence we only have to check that 
\begin{equation}\label{con}
|y-y_0| \le C_0\, |x-y|.
\end{equation}
We have $c_0 \,|y-y_0|= c_0\,t_0 = |y|.$
On the other hand, the condition $y \notin K$ is exactly $|y|< \sqrt{2} \,|y'|$ and clearly $|x-y| \ge |y'|.$  Therefore \eqref{con} holds with an absolute constant $C_0.$

\section{The defining function for $D_t$ }\label{def}

In this section we prove that the function $f$ defined by \eqref{deffun} is a defining function of $D_t$ of class $C^{1+\gamma}.$ Our assumption now is that the CDE has a solution $X(\cdot,t)$ for $t$ in an interval $(-T;T)$ and that $D_t$ is the domain with $\partial D_t=
X(\partial D_0,t)$  which has been defined in subsection 2.1. The field defined by \eqref{velocity} has a flow map \eqref{flowmap} whose restriction to $\partial D_0$ is precisely the solution of the CDE.

Taking the gradient in \eqref{deffun} we get, for $x \notin \partial D_t,$
\begin{equation}\label{gradFi}
\nabla \Phi(x,t) = \det \nabla X (X^{-1}(x,t),t) \, \nabla \varphi(x,t) + \nabla \left(\det \nabla X (X^{-1}(x,t),t)\right) \varphi(x,t).
\end{equation}
In \cite[Section 8]{BGLV} it was shown that $\nabla X^{-1}(\cdot,t)$ satisfies a H\"older condition of order $\gamma$ on the open set
$\R^n \setminus \partial D_t$ (but may have jumps at $\partial D_t$). What remains to be proved is that $\nabla \Phi(\cdot,t)$ extends
continuously to $\partial D_t.$ This is straightforward for the first term in the right hand side of \eqref{gradFi}, just by the jump formulas
\eqref{limdins} and \eqref{limfora}. We have
\begin{equation*}\label{priter}
\lim_{\R^n \setminus \partial D_t  \ni y \to x} \det \nabla X (X^{-1}(y,t),t) \, \nabla \varphi(y,t)  = \frac{|\nabla \varphi_0(X^{-1}(x,t)) |}{\operatorname{det}D(x)} \,\vec{n}(x),
\end{equation*}
where $D(x)$ is the differential at $x \in \partial D_t$ of $X^{-1}(\cdot,t)$ viewed as a differentiable mapping from $\partial D_t$ into 
$\partial D_0.$

The second term in the right hand side of \eqref{gradFi} tends to $0$ as $x$ approaches a point in $\partial D_t.$ Proving this requires some work. For the sake of simplicity of notation let us consider positive times $t$ less than $T.$ Since $X(\cdot,t)$ is a continuously differentiable function of $t$ with values in the Banach space $C^{1+\gamma}(\partial D_0, \R^n),$ the constants $q(D_s)$ determining
the $C^{1+\gamma}$ smoothness of the boundary of $D_s$ are uniformly bounded for $0\le s \le t.$ Hence
\begin{align}\label{gradfu}
||\nabla v (\cdot,s)||_{\infty} \le C(t), \quad 0 \le s \le t,
\\ \label{gradfhol}
||\nabla v (\cdot,s)||_{\gamma, D_s} + ||\nabla v (\cdot,s)||_{\gamma, \R^n \setminus \overline{D_s}}  \le C(t), \quad 0 \le s \le t,
\end{align}
where $C(t)$ denotes here and in the sequel a positive constant depending on $t$ but not on $s \in [0,t].$
Inequality \eqref{gradfu} follows from the fact, already mentioned, that standard even convolution Calder\'on-Zygmund operators are bounded on characteristic functions of $C^{1+\gamma}$ domains with bounds controlled by the constants giving the smoothness of the domain (see, for instance, \eqref{login}). Inequality \eqref{gradfhol} has appeared in the literature several times with various degrees of generality, as we mentioned in the previous section, where a complete proof was presented.  In \cite{MOV} the reader will find another accessible proof independent of Vasin's lemma. The constants are not logarithmic, but this is not relevant here. The statement is that if $T$ is an even  smooth (of class $C^1$) convolution homogeneous Calder\'on-Zygmund operator and $D$ a domain with boundary of class $C^{1+\gamma}, \; 0<\gamma<1,$ then $T(\chi_D)$ satisfies a H\"older condition of order $\gamma$ in $D$ and in $\R^n \setminus \overline{D}.$

As we said in section \ref{out} one applies \eqref{va} to the second order derivatives of the field $v$ to conclude that
\begin{equation}\label{sdv}
|\partial_j \partial_k v(x,s) | \le C(t) \, \operatorname{dist}(x,\partial D_s)^{\gamma-1}, \quad x \notin \partial D_s, \quad 0 \le s \le t, \quad 1\le j,k \le n.
\end{equation}
See \eqref{secder} and \eqref{disbou}.

Combining \eqref{Xbi}, the analogous inequality with $\nabla X(\cdot,t)$ replaced by $\nabla X^{-1}(\cdot,t)$ and \eqref{gradfu} we get
\begin{equation}\label{gradXsd}
C(t)^{-1} \le \|\nabla X(\cdot,s)\|_{\infty} \le C(t), \quad 0 \le s \le t.
\end{equation}
Therefore $X(\cdot,s)$ is a bilipschitz homeomorphism of $\R^n$ and consequently, for all $\alpha \in R^n,$
\begin{equation}\label{dist}
C(t)^{-1} \operatorname{dist}(\alpha,\partial D_0) \le \operatorname{dist}(X(\alpha,s),\partial D_s) \le C(t) \operatorname{dist}(\alpha,\partial D_0), \quad 0 \le s \le t.
\end{equation}

Now let us turn to the second term in the right hand side of \eqref{gradFi}
\begin{equation}\label{ster}
II(x)=\nabla \left(\det \nabla X (X^{-1}(x,t),t)\right) \varphi(x,t)= \varphi_0(\alpha) \,\nabla_x J(\alpha,t),
\end{equation}
where we have set $x=X(\alpha,t)$ and $J(\alpha,t)=\det \nabla X (\alpha,t).$ The jacobian satisfies
\begin{equation*}\label{eqjac}
\frac{d}{dt} J(\alpha,t) = \operatorname{div}v(X(\alpha,t),t) \,J(\alpha,t) 
\end{equation*}
and so
\begin{equation*}\label{jac}
J(\alpha,t) = \exp \int_0^t \operatorname{div}v(X(\alpha,s),s) \,ds.
\end{equation*}
Hence $\nabla_x J(\alpha,t)$ is
\begin{equation}\label{gradjac}
\begin{split}
 \left(\exp \int_0^t \operatorname{div}v(X(\alpha,s),s)  \,ds\right) \, \left(\int_0^t \operatorname{div} \left((\nabla v)^t(X(\alpha,s),s) \right)\, \nabla X(\alpha,s)\,ds\right) \, \nabla X^{-1}(x,t), 
\end{split}
\end{equation}
where the divergence of a matrix is the vector with components the divergence of rows. Combining \eqref{gradfu}, \eqref{sdv}, \eqref{gradXsd}, \eqref{dist}, \eqref{ster} and \eqref{gradjac} we get
\begin{equation*}\label{estsec}
\begin{split}
|II(x)| & \le C(t)\, |\varphi_0(\alpha)| \, \int_0^t  \operatorname{dist}(X(\alpha,s),\partial D_s)^{\gamma-1}\,ds\\*[7pt] 
&\le C(t)\, |\varphi_0(\alpha)| \,  \operatorname{dist}(\alpha,\partial D_0)^{\gamma-1}\\*[7pt] 
& \le C(t)\,  \operatorname{dist}(\alpha,\partial D_0)^{\gamma}.
\end{split}
\end{equation*}
If $ \operatorname{dist}(x,\partial D_t) \to 0$ then $\operatorname{dist}(\alpha,\partial D_0)\to 0$ and thus
 $II(x) \to 0.$

\section{The commutators.}\label{comm}
The material derivative $\frac{D}{Dt}=\partial_t + v \cdot \nabla$ of the defining function of the previous section is
\begin{equation*}\label{matderFi}
\begin{split}
\frac{ D}{Dt} \Phi(x) &=\frac{d}{dt} \big(J(\alpha,t)\,\varphi_0(\alpha)\big) =  \operatorname{div}v(X(\alpha,t),t)\,J(\alpha,t)\,\varphi_0(\alpha) \\*[7pt] 
& = \operatorname{div}v(x,t)\,\Phi(x,t),
\end{split}
\end{equation*}
which proves \eqref{matd}.
Taking derivatives in the equation above and rearranging terms one obtains
\begin{equation}\label{matdgrad}
\frac{D }{Dt} \nabla \Phi = \nabla(\operatorname{div}v) \Phi + (\operatorname{div}v) \nabla \Phi- (\nabla v)^t (\nabla \Phi).
\end{equation}
The first term tends to $0$ at the boundary of $D_t,$ by  \eqref{sdv}. This section is devoted to proving that the second term in the right hand side, namely $ (\operatorname{div}v) \nabla \Phi- (\nabla v)^t (\nabla \Phi),$ is a sum of $n-1$ terms, each of which is a difference of boundary commutators. 
%Since the second term depends linearly on $v$ it is enough to deal with the matrices of the standard basis. We fix a row $i$ and a column $j$, $1\le i,j \le n,$ and take the matrix $L_{ij}$ which has entries $0$ except at the $i$-th row and the $j$-th column, where there is a $1.$
%Let $\chi$ stand for $\chi_{D_t},$ to simplify the writing. Then the velocity field is
%\begin{equation*}\label{velij}
%v=L_{ij}(\nabla N)\star \chi=(0,\dotsc,0,\partial_{j}N \star \chi,0\dotsc,0),
%\end{equation*}
%where $\partial_{j}N \star \chi$ is at the $i$-th component.  Since $\operatorname{div}v=\partial_{ij}N \star \chi$ the matrix $(\operatorname{div}v) I- (\nabla v)^t$ has non zero entries only on the diagonal and on the $i$-the column. The entries on the diagonal
%are $\partial_{ij}N \star \chi$ except at the $i$-th row and $i$-th column where one has $0.$ The entries in the $i$th column are
%$-\partial_{1j}N\star \chi, -\partial_{2j}N\star \chi, \cdots, -\partial_{nj}N\star \chi$ except for the entry on the $i$-th row which is $0.$
%Here is what one has for $n=3, i=j=1,$
%\vspace{0.05cm}
%
%\begin{equation*}\label{ij1}
%\left(\begin{matrix}
%0&\;0&0\\*[12pt]
%-\partial_{21}N*\chi &\; \partial_{11}N*\chi &0\\*[12pt]
%-\partial_{31}N*\chi &\;0 &\partial_{11}N*\chi
%\end{matrix}
%\right)
%\end{equation*}
%\vspace{0.05cm}
%and for  $n=3, i=2, j=1,$
%\begin{equation*}\label{i1j2}
%\left(\begin{matrix}
%\partial_{21}N*\chi&\; -\partial_{11}N*\chi   &0\\*[12pt]
%0 &\;0&0\\*[12pt]
%0&\;-\partial_{31}N*\chi &\;\partial_{21}N*\chi
%\end{matrix}
%\right)
%\end{equation*}
It clearly suffices to prove that each coordinate is a sum of $n-1$ differences of boundary commutators. We present the details for the first coordinate, which is
\begin{equation}\label{coor}
\partial_2 v_2 \,\partial_1 \Phi - \partial_1 v_2 \,\partial_2 \Phi +\dots+\partial_n v_n\, \partial_1 \Phi - \partial_1 v_n \,\partial_n \Phi. 
\end{equation}
Let us work with the first term $\partial_2 v_2 \,\partial_1 \Phi - \partial_1 v_2 \,\partial_2 \Phi .$ The others are treated similarly. The preceding expression is evaluated at $(x,t)$ with $x \in \partial D_t.$ To lighten the notation we set $D=D_t,$ so that $t$ is fixed, and $\chi=\chi_{D_t}.$ Recall  that $v(\cdot,t)=k\star \chi$ and so $$v_j(\cdot,t)=k_j\star \chi, \; 1 \le j \le n.$$ By \eqref{derker} we have in the distributions sense
\begin{equation*}\label{}
\partial_j v_j(\cdot,t) = \partial_j  k_j\star \chi= \operatorname{p.v.} \partial_j k_j \star \chi+ c_j  \chi, \quad 1 \le j \le n.
\end{equation*}
where $c_j =\int_{|\xi|=1} k_j (\xi) \xi_j \, d\sigma(\xi).$
Thus
\begin{equation*}\label{}
\partial_2 v_2(\cdot,t) \,\partial_1 \Phi(\cdot,t)= \left(\partial_2  k_2\star \chi \right)(\cdot) \,\partial_1 \Phi(\cdot,t)= \operatorname{p.v.} \left(\partial_2 k_2 \star \chi\right)\,\partial_1 \Phi + c_2  \chi  \partial_1 \Phi
\end{equation*}
and
\begin{equation*}\label{}
 \partial_2  k_2\star \left(\chi \partial_1 \Phi \right)  =   \operatorname{p.v.} \partial_2 k_2 \star  \left(\chi \partial_1 \Phi\right)  + c_2  \chi  \partial_1 \Phi,
\end{equation*}
which yields
\begin{equation}\label{dosdos}
\partial_2 v_2(\cdot,t) \,\partial_1 \Phi(\cdot,t)= \operatorname{p.v.} \left(\partial_2 k_2 \star \chi\right)\,\partial_1 \Phi -\operatorname{p.v.} \partial_2 k_2 \star  \left(\chi \partial_1 \Phi\right) +\partial_2  k_2\star \left(\chi \partial_1 \Phi \right).
\end{equation}
Similarly
\begin{equation}\label{udos}
\partial_1 v_2(\cdot,t) \,\partial_2 \Phi(\cdot,t)= \operatorname{p.v.} \left(\partial_1 k_2 \star \chi\right)\,\partial_2 \Phi -\operatorname{p.v.} \partial_1 k_2 \star  \left(\chi \partial_2 \Phi\right) +\partial_1  k_2\star \left(\chi \partial_2 \Phi \right).
\end{equation}
Since $\chi \partial_j \Phi =\partial_j (\chi \Phi), \; 1\le j\le n,$ we have
$$
\partial_2  k_2\star \left(\chi \partial_1 \Phi \right)=\partial_1  k_2\star \left(\chi \partial_2 \Phi \right)
$$
and subtracting \eqref{udos} from \eqref{dosdos} yields
\begin{equation}\label{pcoor}
\begin{split}
\partial_2 v_2(\cdot,t) \,\partial_1 \Phi(\cdot,t)-\partial_1 v_2(\cdot,t) \,\partial_2 \Phi(\cdot,t) &= \operatorname{p.v.} \left(\partial_2 k_2 \star \chi\right)\,\partial_1 \Phi -\operatorname{p.v.} \partial_2 k_2 \star  \left(\chi \partial_1 \Phi\right) \\*[5pt]
&- \big(\operatorname{p.v.} \left(\partial_1 k_2 \star \chi\right)\,\partial_2 \Phi -\operatorname{p.v.} \partial_1 k_2 \star  \left(\chi \partial_2 \Phi\right)\big),
\end{split}
\end{equation}
which is the difference of two solid commutators. Here we are using the term ``solid" to indicate that the integration is taken with respect to $n$-dimensional Lebesgue measure.
Our next task is to bring the solid commutators to the boundary.

Formula \eqref{pcoor} is an identity between distributions and is not a priori obvious that the principal value integrals exist at boundary 
points. The same can be said about the principal values on the boundary which appear in the calculation below. That they do exist in our context is a routine argument, which we postpone to the appendix. 
 
 Let $x\in \partial D.$  Given $\ep >0$ set $D_\ep = D\setminus \overline{B(x,\ep)}.$   
% The odd kernel $k_2$ has homogeneity $-(n-1)$ and defines a principal value singular integral of Calder\'on-Zygmund type with respect to the measure $d \sigma$ on $\partial D.$  This operator is well defined at each $x \in \partial D$ when applied to a function satisfying a H\"older condition of any positive order. 
 By the divergence theorem
 \begin{equation*}\label{com3}
 \begin{split}
& \big(\operatorname{p.v.} \partial_2 k_2 \star (\chi \partial_1 \Phi)  \big) (x)=  \lim_{\ep \to 0} \int_{D_\ep} \partial_2 k_2 (x-y) \partial_1\Phi(y)\,dy \\*[5pt]
& = -  \lim_{\ep \to 0} \int_{\partial D_\ep} k_2(x-y) \partial_1\Phi(y) n_2(y)\,d\sigma(y) + \int_{D} k_2(x-y) \partial_{21} \Phi(y)\,dy\\*[5pt]
& = -\operatorname{p.v.} \int_{\partial D} k_2(x-y) \partial_1\Phi(y) n_2(y)\,d\sigma(y) + \int_{D} k_2(x-y) \partial_{12} \Phi(y)\,dy\\*[5pt]
&+ \lim_{\ep \to 0} \int_{\partial B(x,\ep)\cap D} k_2(x-y) \partial_1\Phi(y) n_2(y)\,d\sigma_\ep(y),
\end{split}
\end{equation*}
where $\sigma_\ep$ is the surface measure on $\partial B(x,\ep).$ We do not need to compute explicitly
the term $$ \lim_{\ep \to 0} \int_{\partial B(x,\ep)\cap D} k_2(x-y) \partial_1\Phi(y) n_2(y)\,d\sigma_\ep(y),$$ nor to worry about the second order derivative of $\Phi$ which has appeared, because they will eventually cancel out (a routine regularization argument takes care of
the actual presence of the second derivatives of $\Phi$).

We turn now to the computation of $\big(\operatorname{p.v.} \partial_{12}N \star \chi \big) (x).$ We have
\begin{equation*}\label{com4}
 \begin{split}
 \big(\operatorname{p.v.} \partial_{2} k_2 \star \chi  \big) (x) &=  \lim_{\ep \to 0} \int_{D_\ep} \partial_{2}k_2(x-y) \,dy \\*[5pt]
& = -  \lim_{\ep \to 0} \int_{\partial D_\ep} k_2(x-y)  n_2(y)\,d\sigma(y) \\*[5pt]
& = -\operatorname{p.v.} \int_{\partial D} k_2(x-y)  n_2(y)\,d\sigma(y) \\*[5pt] 
&+\lim_{\ep \to 0} \int_{\partial B(x,\ep)\cap D} k_2(x-y)  n_2(y)\,d\sigma_\ep(y).
\end{split}
\end{equation*}
Therefore
\begin{equation}\label{com5}
 \begin{split}
 & \big(\operatorname{p.v.} \partial_{2}k_2 \star (\chi \partial_1 \Phi)  \big) (x)-  \big(\operatorname{p.v.} \partial_{2}k_2 \star \chi  \big) (x) \partial_1 \Phi(x)\\*[5pt]
& = \operatorname{p.v.} \int_{\partial D} k_2(x-y)  n_2(y)\,d\sigma(y) \, \partial_1 \Phi(x)- \operatorname{p.v.} \int_{\partial D} k_2(x-y) \partial_1\Phi(y) n_2(y)\,d\sigma(y) \\*[5pt]
& +\int_{D} k_2(x-y) \partial_{21} \Phi(y)\,dy,
\end{split}
\end{equation}
since
 \begin{equation*}\label{com6} 
\lim_{\ep \to 0}  \int_{\partial B(x,\ep) \cap D} k_2(x-y)\big(\partial_1 \Phi(y)-\partial_1 \Phi(x)\big) d\sigma_\ep (y)=0,
 \end{equation*}
because $k_2$ is homogeneous of order $-(n-1)$ and $\partial_1 \Phi$ is continuous at $x.$ The conclusion is that the solid commutator in the left hand side of  \eqref{com5} is a boundary commutator plus and additional term involving second order derivatives of $\Phi.$  This term will disappear soon and in the final formulas no second derivatives of $\Phi$ are present, so that the $C^{1+\gamma}$ condition on $\Phi$ is enough.

Proceeding in a similar way we find
\begin{equation}\label{com7}
 \begin{split}
 & \big(\operatorname{p.v.} \partial_1 k_2 \star (\chi \partial_2 \Phi)  \big) (x)-  \big(\operatorname{p.v.} \partial_{1} k_2 \star \chi  \big) (x) \partial_2 \Phi(x)\\*[5pt]
& = \operatorname{p.v.} \int_{\partial D} k_2(x-y)  n_1(y)\,d\sigma(y) \, \partial_2 \Phi(x)-\operatorname{p.v.} \int_{\partial D} k_2(x-y) \partial_2\Phi(y) n_1(y)\,d\sigma(y)\\*[5pt]
& +\int_{D} k_2(x-y) \partial_{12} \Phi(y)\,dy,
\end{split}
\end{equation}
Subtracting \eqref{com7} from \eqref{com5} we see that the difference of the two solid commutators in \eqref{pcoor} is exactly, on the boundary of $D,$ a difference of boundary commutators.  Hence $\big((\operatorname{div}v) I- (\nabla v)^t \big)(\nabla \Phi)$ 
is a sum of $n-1$ terms, each being a difference of two vector valued boundary commutators.

\section{H\"older estimate of differences of  boundary commutators.}\label{hest}
We keep the notation of the previous section $D=D_t, \,\nabla \Phi=\nabla \Phi(\cdot,t),$ with $t$ fixed. Our goal is to estimate the H\"older semi-norm of order $\gamma$ on $\partial D$ of the difference of two boundary commutators. For instance,
\begin{equation*}\label{com8} 
\begin{split}
&DB(x):=\\*[7pt]
&\operatorname{p.v.} \int_{\partial D} k_2(x-y) \partial_2\Phi(y) n_1(y)\,d\sigma(y)- \operatorname{p.v.} \int_{\partial D} k_2(x-y)  n_1(y)\,d\sigma(y) \,\partial_2 \Phi(x) \\*[7pt]
& - \Big(\operatorname{p.v.} \int_{\partial D} k_2(x-y) \partial_1\Phi(y) n_2(y)\,d\sigma(y)- \operatorname{p.v.} \int_{\partial D} k_2(x-y)  n_2(y)\,d\sigma(y)\, \partial_1 \Phi(x) \Big).
\end{split}
\end{equation*}
The general case follows immediately by the same arguments.
The strategy consists in exploiting the fact that $DB(x)$ is also, for $x \in \partial D,$ a difference $DS(x)$ of two solid commutators, as we checked in the previous section.
That is, $DB(x)$ for $ x\in \partial D$ is identical to
\begin{equation*}\label{com9} 
\begin{split}
&DS(x)=DS(\Phi)(x):= \\*[7pt]
&= \big( \operatorname{p.v.} \partial_2 k_2 \star \left(\chi  \partial_1 \Phi \right) -\left(\operatorname{p.v.} \partial_{2} k_2 \star \chi \right) \partial_1 \Phi \big)- \big(\partial_1 k_2 \star \left(\chi  \partial_2 \Phi \right)-\left(\partial_{1}k_2\star \chi \right)  \partial_2 \Phi\big)  .
\end{split}
\end{equation*}
By \cite[Corollary, p.24 and Lemma p.26]{BC}, estimating each commutator separately, we have 
$\|DS\|_{\gamma, \R^n} \le C_n\,||\nabla v(\cdot,t)||_{\infty} \,\|\nabla \Phi\|_{\gamma, \R^n},$ which is not good enough, because we need $\|\nabla \Phi\|_{\gamma, \partial D}$ in place of $\|\nabla \Phi\|_{\gamma, \R^n}.$

We now consider the jet 
\begin{equation*}\label{jet} 
(0,\partial_1 \Phi, \dots, \partial_n \Phi)
\end{equation*}
on $\partial D.$  By Whitney's extension theorem \cite[Chapter VI, p.177]{S} there exists $\Psi$ of class $C^{1+\gamma}(\R^n)$ such that 
$\Psi=0$ and $\nabla \Psi= \nabla \Phi$ on $\partial D,$ satisfying
\begin{equation}\label{jet1} 
\|\nabla \Psi\|_{\gamma, \R^n} \le C_n\, \left(\|\nabla \Phi\|_{\gamma, \partial D}+\sup \{\frac{|\nabla \Phi(x)\cdot (y-x)|}{|y-x|^{1+\gamma}}: y\neq x, \;y,x\in \partial D\} \right).
\end{equation}
This precise estimate is not stated explicitly in Stein's book but it follows from the proof. 
%The right hand side is part of the norm of the jet \eqref{jet}, namely
%\begin{equation}\label{jet1} 
%\|\nabla \Phi\|_{\infty, \partial D}+\|\nabla \Phi\|_{\gamma, \partial D}+\sup \{\frac{|\nabla \Phi(x)\cdot (y-x)|}{|y-x|^{1+\gamma}}: y\neq x, y,x\in \partial D\}.
%\end{equation}
We claim that
\begin{equation}\label{jet2} 
\sup \{\frac{|\nabla \Phi(x)\cdot (y-x)|}{|y-x|^{1+\gamma}}: y\neq x, \;y,x\in \partial D\} \le 2^{(1+\gamma)/2} \|\nabla \Phi\|_{\gamma, \partial D}.
\end{equation}
We postpone the proof of the claim and we complete the estimate of $\|DB\|_{\gamma, \partial D}.$

The extension $\Psi$ of the jet $(0,\partial_1 \Phi, \dots, \partial_n \Phi)$ on $\partial D$ , given by Whitney's extension theorem, satisfies, in view of \eqref{jet1} and  \eqref{jet2} , 
$$\|\nabla \Psi\|_{\gamma,\R^n} \le C_{n,\gamma}\, \|\nabla \Phi\|_{\gamma,\partial D}.$$
  Since $\nabla \Psi =\nabla \Phi$ on $\partial D$
the differences of solid commutators $DS(\Phi)$ and $DS(\Psi)$ are equal on $\partial D.$ Thus
\begin{equation*}\label{diff} 
\begin{split}
\|DB\|_{\gamma, \partial D} &=\|DS(\Psi)\|_{\gamma, \partial D} \le \|DS(\Psi)\|_{\gamma, \R^n}\\*[7pt] 
&\le C_n \,(||\nabla v(\cdot,t)||_{\infty}+1) \,\|\nabla \Psi\|_{\gamma, \R^n}\le C_n \,(||\nabla v(\cdot,t)||_{\infty}+1) \,\|\nabla \Phi\|_{\gamma,\partial D}.
\end{split}
\end{equation*}
This can be used to prove the a priori estimate \eqref{ap2} as in \cite{BC}.

We turn now to the proof of the claim \eqref{jet2}. 
%We use a variant of the argument for the proof of lemma 6.4 of \cite{BGLV}, which we discuss in detail for the reader's convenience. 
Fix a point $x \in \partial D.$ Assume without loss of generality that $x=0$ and
$\nabla\Phi(0)= (0,\dotsc,0, \partial_n \Phi(0))$, $\partial_n
\Phi(0) > 0.$  Define $\delta=\delta(x)$ by 
$$
\delta^{-\gamma}= 2\,\frac{\|\nabla \Phi\|_{\gamma,\partial D}}{|\nabla \Phi(0)|}.
$$
This choice of $\delta$ implies that the normal vector $\nabla \Phi(y)$ remains for $y\in B(0,\delta)\cap \partial D$ in the ball
$B(\nabla \Phi(0), |\nabla \Phi(0)| /2).$ Indeed
\begin{equation*}
|\nabla \Phi(y)-\nabla \Phi(0)| \le \|\nabla \Phi\|_{\gamma,\partial D} \, \delta^{\gamma}  = \frac{ |\nabla \Phi(0)|}{2}.
\end{equation*}
Then given $y\in B(0,\delta)\cap \partial D$ the tangent hyperplane to $\partial D$ at $y$ forms an angle less than $30$ degrees with the horizontal plane and thus $\partial D$ is the graph of a function $y_n=\varphi(y'_n)$ which satisfies a Lipschitz condition with constant less than $1.$ Here we have used the standard notation $y=(y',y_n), \; y'=(y_1, \dots, y_{n-1}).$ The function $\varphi$ is defined in the open set $U$ which is the projection of $B(0,\delta)\cap \partial D$ into $\R^{n-1}$ defined by $y \rightarrow y'.$ 
By the implicit function theorem $\varphi$ is of class $C^{1+\gamma}$ in its domain.

Note that for each $y\in \partial D\cap B(0,\delta),$ the segment $\{t\,y' : 0\le t\le 1 \}$ is contained in $U,$ as an elementary argument shows. The mean value theorem on that segment for the function $t \to \varphi(ty')$ yields
$$
 \frac{|\nabla \Phi(0)\cdot y|}{|y|^{1+\gamma}}= 
\frac{|\nabla \Phi(0)| |\varphi(y')|}{|y|^{1+\gamma}}\le  \frac{|\nabla \Phi(0)|}{|y|^{1+\gamma}} \sup \{|\nabla \varphi(z')| : z' \in U, \; |z'| \le |y'|\}
\,|y'|
$$

By implicit differentiation
$$
\partial_j \varphi(z')=- \frac{\partial_j \Phi(z',\varphi(z'))}{\partial_n
\Phi(z',\varphi(z'))}, \quad 1 \le j \le n-1,
$$
and so, recalling that $\partial_j \Phi(0)=0, \, \; 1 \le j \le n-1,$ and that $z=(z',\varphi(z')),$

$$
|\nabla\varphi(z')| \le
\frac{\|\nabla\Phi\|_{\gamma,\partial D} \,|z|^\gamma}{|\partial_n \Phi(z)|} \le \frac{2}{|\nabla \Phi(0)|}\, \|\nabla\Phi\|_{\gamma,\partial D}\;2^{\gamma/2}\,|z'|^\gamma,
\quad |z'|  \le |y'|,
$$
because
\begin{equation*}
\begin{split}
|\partial_n \Phi(z)| & \ge |\partial_n\Phi(0)| -|\partial_n \Phi(z) - \partial_n \Phi(0)|  \\*[7pt]
& \ge |\nabla \Phi(0)|-\|\nabla \Phi\|_{\gamma,\partial D} \, \delta^{\gamma} = \frac{ |\nabla \Phi(0)|}{2}
\end{split}
\end{equation*}
and
$$
|z|=(|z'|^2+ \varphi(z')^2)^{1/2} \le  \sqrt{2} \,|z'|.
$$
Thus
$$
\frac{|\nabla \Phi(0)\cdot y|}{|y|^{1+\gamma}} \le 2^{1+\gamma/2} \,\|\nabla\Phi\|_{\gamma,\partial D}, \quad y\in\partial D\cap B(0,\delta).
$$
If $y \in \partial D \setminus B(0,\delta)$
$$
\frac{|\nabla \Phi(0)\cdot y|}{|y|^{1+\gamma}} \le \frac{|\nabla \Phi(0)|}{|y|^{\gamma}}\le \frac{|\nabla \Phi(0)|}{\delta^{\gamma}} =2\, \|\nabla \Phi\|_{\gamma,\partial D},
$$
which completes the proof of \eqref{jet2}.

\section{Domains enclosed by ellipses as Cauchy patches.}\label{ell}
In this section we consider the transport equation in the plane given by the Cauchy kernel
\begin{equation}\label{transcau2}
\begin{split}
&\partial_{t}\rho(z,t)+v(z,t)\cdot \nabla \rho(z,t)=0,\\
&v(z,t)=\left(\frac{1}{\pi z} * \rho (\cdot, t)\right)(z),\\
&\rho(z,0)=\chi_{D_{0}}(z),
\end{split}
\end{equation}
$z=x+i y \in \C=\R^2$ and $t\in \R.$ Note that we have changed the normalization of the velocity field in  \eqref{transcau} by a factor of $2.$

We take the initial patch to be the domain enclosed by an ellipse
\begin{equation*}\label{ell1}
E_{0}=\left\{ (x,y)\in \mathbb{R}^{2}:  \frac{x^{2}}{a_0^{2}}+\frac{y^{2}}{b_0^{2}}<1\right\}.
\end{equation*}
We will show that the solution provided by the Theorem is of the form $\chi_{E_t}(z)$ with
\begin{equation*}\label{ellt}
E_{t}=\left\{ (x,y)\in \mathbb{R}^{2}:  \frac{x^{2}}{a(t)^{2}}+\frac{y^{2}}{b(t)^{2}}<1\right\},
\end{equation*}
and
\begin{align}\label{ida}
a(t)= a_0 \frac{(a_0+b_0)\,e^{2t}}{b_0+a_0\,e^{2t}},\quad t\in\R \\\label{idb}
b(t)=b_0 \frac{(a_0+b_0)}{b_0+a_0\,e^{2t}},\quad t\in\R.
\end{align}
As $t \to \infty,\,$ $a(t) \to a_0+b_0$ and $b(t) \to 0,$  so that the ellipse at time $t$ degenerates into the segment $[-(a_0+b_0),a_0+b_0]$ as $t \to +\infty$ and into the segment $i [-(a_0+b_0),a_0+b_0]$ on the vertical axis as $t \to -\infty.$ 

Since \eqref{transcau2} is not rotation invariant, one has to consider also the case of an initial patch given by the domain enclosed by a tilted ellipse
\begin{equation*}\label{tilted2}
E(a,b,\theta)=e^{i\theta}
\left\{ (x,y)\in \mathbb{R}^{2}:  \frac{x^{2}}{a^{2}}+\frac{y^{2}}{b^{2}}<1\right\}.
\end{equation*}
In this case the straight line containing the semi-axis of length $a$ makes an angle $\theta$ with the horizontal axis and we take $0 < \theta < \pi/2.$

Assume that the initial patch is $E_0=E(a_0,b_0,\theta_0).$ Then we will show  that the solution given by the theorem is $\chi_{E_t}$  with $E_t=E(a(t), b(t),\theta(t)),$ where $a(t), b(t)$ and $\theta(t)$ are the unique solutions of the system
\begin{equation}\label{sys}
\begin{split}
a'(t)&=\frac{2}{a_0+b_0}\, a(t)\,b(t)\, \cos(2\theta(t)) \\*[7pt]
b'(t)&= -\frac{2}{a_0+b_0}\, a(t)\,b(t)\, \cos(2\theta(t)) \\*[7pt]
\theta'(t)&= -\frac{2}{a_0+b_0}\, \frac{a(t)\,b(t)}{a(t)-b(t)}\, \sin(2\theta(t)),
\end{split}
\end{equation}
with initial conditions $a(0)=a_0, \; b(0)=b_0, \; \theta(0)=\theta_0.$

We start the proof by assuming that the patch $D_t$ of the weak solution provided by the theorem is indeed $E_t.$ Let $z(t)$ be the trajectory of the particle that at time $0$ is at $z(0) \in\partial E_0.$ Then
\begin{equation}\label{flow}
\frac{dz}{dt}=v(z(t),t), \quad\quad\quad z(0) \in \partial E_0,
\end{equation}
and $v$ is the velocity field of \eqref{transcau2}. It is a well-known fact that $v$ can be explicitly computed \cite{HMV}. One has
\begin{equation}\label{cauell}
v(z,t)=\left(\frac{1}{\pi z} \star \chi_{E_t}\right)(z)=\overline{z}-q(t)\, e^{-2\theta(t)}z, \quad z \in E_t, \quad q(t)=\frac{a(t)-b(t)}{a(t)+b(t)}.
\end{equation}
Indeed in 	\cite{HMV} only the case $\theta(t)=0$ is dealt with, but the general case follows easily from the behavior under rotations of a convolution with the Cauchy kernel. 

To lighten the notation think that $t$ is fixed and write $a=a(t), b=b(t), \theta=\theta(t), q=q(t), z=z(t)=x(t)+iy(t)=x+iy.$ 
The condition $z(t) \in \partial E_t$ is equivalent to $ e^{-i\theta(t)}z(t) \in \partial E(a(t), b(t),0),$ 
which is
\begin{equation*}\label{}
\frac{\left(x\,\cos(\theta)+y\,\sin(\theta)\right)^2}{a^2}+\frac{(-x\,\sin(\theta)+y\,\cos(\theta))^2}{b^2}=1,
\end{equation*}
and can also be written more concisely as
\begin{equation}\label{eqel}
\frac{{\big\langle z, e^{i\theta} 	\rangle}^2}{a^2}+\frac{{\big\langle z, i\,e^{i\theta} \rangle}^2}{b^2}=1.
\end{equation}
Here we have denoted by $\langle u,v\rangle$ the scalar product of the vectors $u$ and $v.$
Now proceed as follows. Take the derivative in \eqref{eqel} with respect to $t$ and then replace $z'(t)$ by the expression of the field
given by \eqref{cauell}. We get an equation containing $a,b,\theta \;\text{and}\; z,$  which determines $z(t),$ the solution of the CDE. 
This equation is 
\begin{equation}\label{eqzeta}
\begin{split}
0&=\frac{{\big\langle z, e^{i\theta} 	\rangle}}{a^2}\left(\big\langle z, e^{i\theta}\,i\,\theta'-q\,e^{3i\theta}\rangle + \big\langle  \overline{z}, e^{i\theta} \rangle
\right)- \frac{a'}{a^3} {\big\langle z, e^{i\theta} \rangle}^2 \\*[7pt]
&+  \frac{{\big\langle z, i\,e^{i\theta} 	\rangle}}{b^2}\left(\big\langle z, -e^{i\theta}\,\theta'-i\,q\,e^{3i\theta}\rangle + \big\langle \overline{z}, i\,e^{i\theta} \rangle \right)- \frac{b'}{b^3} {\big\langle z, i\,e^{i\theta} \rangle}^2.
\end{split}
\end{equation}

Evaluate at  $z=z(t)=a(t)e^{i\theta(t)}$ (which is a vertex of the ellipse at time $t$). One gets the equation
\begin{equation}\label{eqa}
a'=2 \frac{a\,b}{a+b}\,\cos(	2\theta).
\end{equation}
Evaluating at the other vertex of the ellipse at time $t,$ that is, at $z=z(t)=b(t)\,i\, e^{i\theta(t)},$ yields
\begin{equation}\label{eqb}
b'=-2 \frac{a\,b}{a+b}\,\cos(2\theta).
\end{equation}
Adding \eqref{eqa} and \eqref{eqb} we see that $a+b$ is constant, then equal to $a_0+b_0.$ Thus we have the first two equations in
 \eqref{sys}. 
 
 Before getting the third equation let us solve the case in which the initial ellipse has axes parallel to the coordinate axes ($\theta_0=0$).
 In this case set $\theta(t)=0, \; t \in \R.$
Replacing in \eqref{eqb} $a$ by
$a_0+b_0-b$ and solving we get \eqref{idb} and then \eqref{ida}.

Now take the domain $E_t=E(a(t),b(t),0)$, the vector field 
$$v(z,t)=\left(\frac{1}{\pi z} \star \chi_{E_t}\right)(z)=\overline{z}-q(t)\,z, \quad z \in E_t,$$
and the flow
\begin{equation*}\label{flow2}
\frac{dz}{dt}=\overline{z}-q(t)\,z, \quad\quad\quad z(0) \in \partial E_0,
\end{equation*}
The preceding system is
\begin{equation*}\label{flow3}
\begin{split}
\frac{dx(t)}{dt} &=  \frac{2 b(t)}{a_0+b_0} x(t)\\*[5pt]
\frac{dy(t)}{dt} &= - \frac{2 a(t)}{a_0+b_0} y(t).
\end{split}
\end{equation*}
Then the flow map is linear on $E_0$ and given by a diagonal matrix. Hence the flow preserves the coordinate axes and maps $\partial E_0$ into an ellipse with axes parallel to the coordinate axes enclosing a domain $\tilde{E_t}$ . But \eqref{eqa} and  \eqref{eqb} say exactly that the vertices of $\partial E_t$ belong to $\partial \tilde{E_t}.$ Thus $\tilde{E_t}=E_t$ and so
$\chi_{E_t}$ is the unique weak solution of the Cauchy transport equation in the class of characteristic functions of $C^{1+\gamma}$ domains.

Let us now go back to the general case and obtain a third equation involving $\theta'.$  Impose that the intersection of the ellipse $\partial E_t$ with the positive real axis belongs to the image of  $\partial E_0$  under the flow. In other words replace $z(t)$ in \eqref{eqzeta} by $$((\frac{\cos^2\theta}{a^2}+\frac{\sin^2\theta}{b^2})^{-1/2},0).$$  After a lengthy computation one gets
\begin{equation}\label{eqter}
\theta'=-\frac{2}{a_0+b_0}\, \frac{ab}{a-b} \,\sin(2\theta), 
\end{equation}
provided  $a\neq b.$

We know claim that the system \eqref{sys} has a unique solution defined for all times $t\in \R$ provided $a_0 \neq b_0.$ The case
$a_0=b_0$ corresponds to an initial disc and so to the case $\theta_0=0,$ which has been discussed before. Consider the open set
\begin{equation*}
 \Omega=\{(a,b,\theta) \in \R^3: a>0,\, b>0, \, a\neq b \text{ and } 0<\theta<\frac{\pi}{2}\}. 
\end{equation*}
Clearly a unique solution of the system exists locally in time for any initial condition $(a_0,b_0,\theta_0)\in \Omega,$ because the function giving the system is $C^{\infty}$  in $ \Omega.$ We claim that this solution exists for all times. Assume that the maximal interval of existence is  $(-T,T)$ for some $0<T<\infty.$  By the first two equations of the system \eqref{sys}  $|a'|$ and $|b'|$ are bounded above by  $2(a_0+b_0)$ and hence the limits $\lim_{t\rightarrow T^{}}a(t)=a(T^{})$  and $\lim_{t\rightarrow T^{}}b(t)=b(T^{})$ exist.
We also have
\begin{equation*}
 \Big|\frac{a'}{a}\Big| \le 2 \quad \text{and}\quad
 \Big|\frac{b'}{b}\Big| \le2
 \end{equation*}
and so
\begin{equation*}
 0<a_0e^{-2T^{}}\le a(T^{}) \le a_0e^{2T^{}} \text{ and } 0<b_0e^{-2T^{}}\le b(T^{}) \le b_0 e^{2T^{}}. 
\end{equation*}

Note that $\theta'(t)$ cannot vanish. Otherwise, by \eqref{eqter},  $\theta(t)=0$ for some $t,$ and in this case we have already checked that the system can be solved for all times. Hence $\theta'$ has constant sign. When $a_0>b_0$  the function $\theta$ decreases and if $a_0<b_0$ the function $\theta$ increases. In any case we have that there exists 
$$
\theta(T^{})=\lim_{t\rightarrow T^{}}\theta(t). 
$$
We cannot have $\theta(T^{})=0$ or $\theta(T^{})=\frac{\pi}{2}$ because we have solved the equation in these cases for all times. For the same reason we cannot have $a(T^{})=b(T^{}).$   Therefore  $(a(T^{}),b(T^{}),\theta(T^{}))\in \Omega$ and we can solve the system  past $T^{},$ which is a contradiction. 

We proceed now to prove that the domain $E_t=E(a(t),b(t),\theta(t))$ enclosed by the ellipse provided by the solution of \eqref{sys}
yields the weak solution $\chi_{E_t}$ of the transport equation \eqref{transcau2} with initial condition $D_0=E_0$. We consider the field 
\eqref{cauell} and the trajectory \eqref{flow} of a particle initially at the boundary point $z(0) \in \partial E_0.$ Since the velocity field is linear in $E_t$ the flow is a linear function of $z(0)\in E_0.$ Thus the initial ellipse $\partial E_0$ is mapped into an ellipse $\partial \tilde{E_t}$
enclosing $\tilde{E_t},$ the image of $E_0$ under the flow map. To show that $\chi_{E_t}$ is a weak solution of the Cauchy transport equation we only need to ascertain that $E_t=\tilde{E_t}.$ But the three equations of \eqref{sys} simply mean that the vertices of $E_t$ and  the intersection of $E_t$ with the horizontal axis are in the image of $\partial E_0$ under the flow map. It is now a simple matter to realize that there is only one ellipse centered at the origin containing those three points. 

A surprising result arises when examining the asymptotic behavior as $t \to \infty$ of the weak solution of the Cauchy transport equation \eqref{transcau2} when the initial condition is $E(a_0,b_0,\theta_0)$ with $a_0\neq b_0$ and $\theta_0 > 0.$ We know that the solution of the system \eqref{sys} never leaves the open set $\Omega.$ In particular $a(t)-b(t)$ has a definite sign determined by the initial condition. Assume for definiteness that $a_0-b_0 >0,$ so that $a(t)-b(t) >0, \;t\in \R$ and hence $\theta(t)$ is a decreasing function. Then
the limit $\theta_\infty=\lim_{t \to \infty} \theta(t)$ exists. 
The system \eqref{sys} readily yields that the function $(a-b) \sin(2\theta)$ has vanishing derivative, so that
\begin{equation}\label{sin}
(a(t)-b(t)) \sin(2\theta(t))=(a_0-b_0) \sin(2 \theta_0), \, t \in \R. 
\end{equation}
Thus
$(a_0+b_0)  \sin(2\theta(t)) \ge (a(t)-b(t)) \sin(2\theta(t))=(a_0-b_0) \sin(2\theta_0)$
 and taking limits 
 $$\sin(2 \theta_\infty) 	\ge \frac{a_0-b_0}{a_0+b_0} \sin(2\theta_0)>0,$$
which means that the limit angle $\theta_\infty$ is positive. In other words, the axes of the ellipses at time $t$ do not approach the coordinate axes.

Assume that $0<\theta_0 	\le \frac{\pi}{4}.$ Since $\theta(t)$ decreases, $0< 2 \theta(t) < \frac{\pi}{2},\; t>0,$  which implies that $a(t)$
 increases and $b(t)$ decreases. By \eqref{eqb}
$$
b(t) = b_0 \exp \int_0^t -\frac{2}{a_0+b_0} a(s) \cos(2\theta(s))\,ds 	\le b_0 \exp\left(-\frac{2 a_0}{a_0+b_0} \cos(2\theta_0) \,t \right),
$$
and so $b_\infty=0,$  provided $\theta_0 < \frac{\pi}{4}.$  If $\theta_0 =\frac{\pi}{4}$  we break the integral above into two pieces, the first between $0$ and $1$ and the second between $1$ and $t.$ We get, for some constant $C$ independent of $t,$
$$
b(t) \le C\, \exp\left(-\frac{2 a_0}{a_0+b_0} \cos(2\theta(1)) \,(t-1) \right), \quad t>1,
$$
which again yields $b_\infty=0.$
By \eqref{sin}
\begin{equation}\label{tetainf}
\sin(2 \theta_\infty)= \frac{a_0-b_0}{a_0+b_0} \sin(2\theta_0),
\end{equation}
which determines the limit angle in terms of the initial data.

\begin{figure}[ht]
\centering
\begin{minipage}{.4\textwidth}
\includegraphics{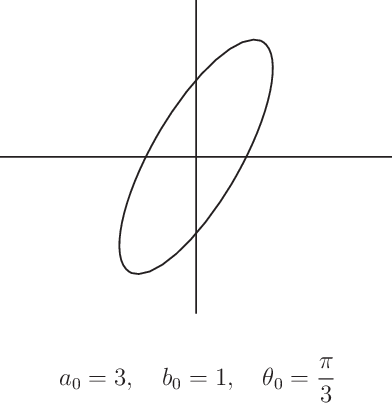}
%\caption{}\label{figure1}
\end{minipage}
\hspace*{1cm} 
\begin{minipage}{.4\textwidth}
\includegraphics{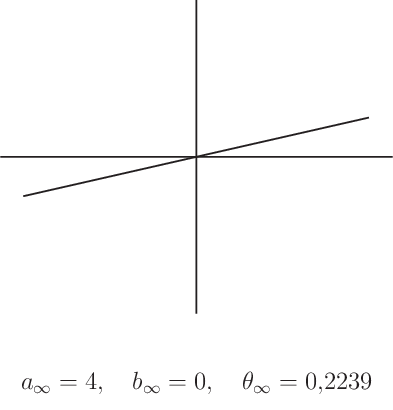}
%\caption{}\label{figure2}
\end{minipage}%
\caption{The initial ellipse and the final segment.}
\label{fig:ellipse}
\end{figure}

Let us turn now to the case $ \frac{\pi}{4} < \theta_0 	< \frac{\pi}{2}.$ In Figure \ref{fig:ellipse} one can see the initial ellipse and the final segment. In view of the first two equations of the system \eqref{sys}
at least for a short time  $a(t)$ decreases and $b(t)$ increases. If
one has $ \frac{\pi}{4} \le \theta_\infty,$ then  $\cos(2\theta(t)) < 0$ for $t>0$ and $a(t)$ decreases and $b(t)$ increases for all times. Integrating the third equation in \eqref{sys} we obtain
\begin{equation*}
\begin{split}
\tan(\theta(t))& =\tan(\theta_0) \exp\left(\frac{-4}{a_0+b_0} \int_0^t \frac{a(s)b(s)}{a(s)-b(s)}\,ds\right)\\*[7pt]
& \le \tan(\theta_0) \exp\left(-\frac{-4 b_0^2 }{a_0^2-b_0^2} \,t\right).
\end{split}
\end{equation*}
Letting $t \to \infty$ we get $\tan(\theta_	\infty)=0,$ which is impossible. Hence $\theta_\infty < \frac{\pi}{4}.$ Then for some
$t_0$ we have $\theta(t_0)< \frac{\pi}{4},$ which brings us into the previous case, in particular to the expression \eqref{tetainf} for the limiting angle $\theta_\infty.$

Arguing similarly with $t\to -\infty$ we get \eqref{tetainf} with $\sin(2 \theta_\infty)$ 
replaced by $\sin(2 \theta_{-\infty}),$  where $\theta_{-\infty}=\lim_{t\to -\infty} \theta(t).$ Thus $\theta_{-\infty}=\frac{\pi}{2}-\theta_\infty.$

The case $a_0<b_0$ is reduced to $a_0>b_0$ by taking conjugates (symmetry with respect to the horizontal axis).
Indeed, \eqref{transcau2} is invariant by taking conjugates, as a simple computation shows. If one has $a_0<b_0$ and an angle $\theta_0,$ the symmetric ellipse has semi-axes $A_0=b_0,B_0=a_0$ and angle $\theta_0'=\frac{\pi}{2}-\theta_0.$

\appendix
\setcounter{secnumdepth}{0}
\section{Appendix: Existence of principal values }\label{app}
The first fact we prove in this section is the following.

\begin{lemma*}\label{epv}
Let $D$ be a bounded domain with boundary of class $C^{1+\gamma},\; 0<\gamma<1.$  Let $L: \R^n\setminus \{0\} \rightarrow \R$ be an even kernel, continuous on $\R^n\setminus \{0\},$  homogeneous of degree $-n,$ which satisfies  cancellation property $\int_{|\xi|=1} L(\xi) \,d\sigma(\xi)=0.$
Then for each $x\in \partial D$ the principal value
\begin{equation*}\label{pv}
\left(\operatorname{p.v.} L\star \chi_D\right) (x)=\lim_{\ep\to 0} \int_{\{	y\in D: |y-x|>\ep\}} L(x-y)\,dy
\end{equation*}
exists.
\end{lemma*}
\begin{proof} Without loss of generality assume that $x=0$, that the tangent hyperplane to $\partial D$ at $0$ is $\{y\in\R^n:y_n=0\}$
and that $r_0>0$ is so small that there exists a function 
$$\varphi \in C^{1+\gamma}(B'(0,2r_0)), \quad B'(0,2r_0)=\{y \in \R^n: |y'|< 2 r_0 \}, \quad y'=(y_1,\dots,y_{n-1})$$
such that $D\cap B(0,r_0)=\{y\in B(0,r_0): y_n < \varphi(y')\}.$  

For $0<r$ set 
$$S_r=\{y \in \R^n: |y|=r\}, \quad S_r^+=\{y \in S_r: y_n >0\}\quad \text{and}\quad S_r^-=\{y \in S_r: y_n <0\}.$$
Since $L$ is even
$$
0=\int_{S_r} L(y) \, d\sigma(y) = \int_{S_r^+} L(y) \, d\sigma(y) +\int_{S_r^-} L(y) \, d\sigma(y) =2 \int_{S_r^-} L(y) \, d\sigma(y). 
$$
Set $H_{-}=\{y\in \R^n : y_n <0\}.$  For $0< \delta <\ep < r_0$ we then have
\begin{equation*}\label{}
\begin{split}
&-\int_{\{y\in D: |y|>\ep\}} L(y)\,dy + \int_{\{y \in D: |y|>\delta\}} L(y)\,dy  \\*[5pt] &=\int_{\{y\in \R^n: \delta<|y|<\ep\}\cap(D\setminus H_{-})} L(y)\,dy 
- \int_{\{y\in \R^n: \delta<|y|<\ep\} \cap (H_{-}\setminus D)} L(y)\,dy.
\end{split}
\end{equation*}
The tangential domains $\left(D\setminus H_{-}\right) \cap B(0,\ep)$ and $\left(H_{-}\setminus D\right) \cap B(0,\ep)$ are very small.  Indeed,
\begin{equation*}\label{}
\begin{split}
\left|\int_{\{y\in \R^n: \delta<|y|<\ep\}\cap(D\setminus H_{-})} L(y)\,dy \right| & \le \int_\delta^\ep \frac{1}{\rho^n} \rho^{n-1}  \sigma\{ \theta \in S^{n-1} : \rho \theta \in D\setminus H_{-} \} \,d\rho\\*[5pt]
 & \le  C\, \int_\delta^{\ep} \rho^{-1+\gamma} \ \,d\rho 	\le \frac{C}{\gamma} \,\ep^\gamma.
\end{split}
\end{equation*}
One obtains in the same way  
$$ \left|\int_{\{y\in \R^n: \delta<|y|<\ep\} \cap (H_{-}\setminus D)} L(y)\,dy \right| \le \frac{C}{\gamma} \,\ep^\gamma$$
and so the proof is complete.
\end{proof}

The second result is the following.

\begin{lemma*}\label{epv2}
Let $D$ be a bounded domain with boundary of class $C^{1+\gamma},\; 0<\gamma<1.$  Let $K: \R^n\setminus \{0\} \rightarrow \R$ be an odd kernel of class $C^1(\R^n\setminus \{0\}),$  homogeneous of degree $-(n-1).$ Let $\varphi$ be a function defined on $\partial D$ satisfying a H\"older condition of some positive order on $\partial D.$
Then for each $x\in \partial D$ and each $1 \le j \le n$ the principal value
\begin{equation*}\label{pv}
\left(\operatorname{p.v.} K\star \varphi n_j d\sigma \right) (x)=\lim_{\ep\to 0} \int_{\{	y\in \partial D: |y-x|>\ep\}} K(x-y) \varphi(y) n_j(y)\,d\sigma(y)
\end{equation*}
exists.
\end{lemma*}
\begin{proof}
It is easy to get rid of $\varphi.$ Indeed
\begin{equation*}\label{}
\begin{split}
&\int_{\{y\in \partial D: \ep<|y-x|\}} K(x-y) \varphi(y) n_j(y)\,d\sigma(y) \\*[5pt]
&=\int_{\{y\in \partial D: \ep<|y-x|\}} K(x-y) (\varphi(y)- \varphi(x)) n_j(y)\,d\sigma(y) \\*[5pt] 
&+ \varphi(x) \int_{\{y\in \partial D: \ep<|y-x|\}} K(x-y) n_j(y)\,d\sigma(y)
\end{split}
\end{equation*}
and the first integral in the right hand side tends as $\ep \to 0$ to the absolutely convergent integral
$$
\int_{\partial D} K(x-y) (\varphi(y)- \varphi(x)) n_j(y)\,d\sigma(y).
$$
Hence we can assume that $\varphi$ is identically $1.$

We can also assume, as in the proof of the previous lemma, that $x=0$, the tangent hyperplane to $\partial D$ at $0$ is $\{y \in \R^n:y_n=0\}$ and the domain $D$ inside $B(0,\ep)$ is  exactly $\{y \in B(0,\ep) : y_n < \varphi(y')\}.$
By the divergence theorem
%\begin{equation*}\label{}
%\begin{split}
%\int_{\{y\in \partial D: \ep<|y|\}} K(x-y)  n_j(y)\,d\sigma(y) =&- \int_{\{y\in D: |y| >\ep\}} \partial_j K(y)\,dy  \\*[5pt]
%&- \int_{\{y\in D: |y|=\ep\}} K(y)  n_j(y)\,d\sigma(y)
%\end{split}
%\end{equation*}
\begin{equation*}\label{}
\begin{split}
\int_{\{y\in \partial D: \ep<|y|\}} K(-y)  n_j(y)\,d\sigma(y) &=- \int_{\{y\in D: \ep < |y| \}} \partial_j K(y)\,dy 
+ \int_{\{y\in D: |y|=\ep\}} K(y)  n_j(y)\,d\sigma(y)\\*[5pt] & =-I+II,
\end{split}
\end{equation*}
where in the last identity one is defining $I$ and $II.$

To apply the previous lemma to $I$ we need to check that $\partial_j K(y),$  which is continuous off the origin, even and homogeneous of degree $-n$ and  has vanishing integral on the unit sphere. By the divergence Theorem
$$
\int_{1< |y|< 2} \partial _j K(y)\,dy= \int_{|y|= 2} K(y)\,n_j(y)\,d\sigma(y)-\int_{|y|=1} K(y)\,n_j(y)\,d\sigma(y),
$$
which is $0,$ since the two integrals over the spheres are the same by homogeneity. Hence, changing to polar coordinates,
$$
0=\int_{1< |y|< 2} \partial _j K(y)\,dy=\log 2 \, \int_{|\theta|=1} K(\theta)\,d\sigma(\theta),
$$
which takes care of $I.$

For the term $II$, set, as before, $H_{-}=\{y\in \R^n : y_n <0\}.$  We then have
\begin{equation*}\label{}
\begin{split}
\int_{\{y\in D: |y|=\ep\}} K(y)  n_j(y)\,d\sigma(y)&= \int_{\{y \in D\setminus H_{-} : |y|= \ep\}} K(y)\,n_j(y)\,d\sigma(y)\\*[5pt]
&+\int_{\{y \in H_{-} : |y|= \ep\}} K(y)\,n_j(y)\,d\sigma(y)\\*[5pt]
&-\int_{\{y \in H_{-}\setminus D: |y|= \ep\}} K(y)\,n_j(y)\,d\sigma(y)\\*[5pt]
%&=I_\ep+I_\ep+I_\ep
\end{split}
\end{equation*}
The first and third terms tend to $0$ with $\ep,$ because the domains of integration are tangential. Indeed,
$$
\sigma \left(\partial B(0,\ep) \cap (D\setminus H_{-})	\right)+\sigma \left(\partial B(0,\ep) \cap (H_{-}\setminus D)\right) \le C \, \ep^{n-1+\gamma}
$$
and so the absolute value of the first and third terms can be estimated by $C\,\ep^\gamma.$

It only remains to note that the second term is independent of $\ep,$ by homogeneity.

\end{proof}

%\vspace{5cm}
\begin{acknowledgements} 
The authors warmly thank Francisco de la Hoz for a numerical simulation of the ellipses' evolution; they acknowledge support by 2017-SGR-395 (Generalitat de Cata\-lunya) and PID2020-112881GB-I00.  Mateu, Orobitg and Verdera are supported by Severo Ochoa and Maria de Maeztu centers of excellence CEX2020-001084-M.
\end{acknowledgements}

\vspace{0.5cm}
{\small
\begin{tabular}{@{}l}
J.C.\ Cantero,\\ Departament de Matem\`{a}tiques, Universitat Aut\`{o}noma de Barcelona.\\
 J.\ Mateu, J.\ Orobitg and J.\ Verdera,\\
Departament de Matem\`{a}tiques, Universitat Aut\`{o}noma de Barcelona,\\ 
Centre de Recerca Matem\`atica, Barcelona, Catalonia.\\
{\it E-mail:} {\tt cantero@mat.uab.cat},\, {\tt mateu@mat.uab.cat},\, {\tt orobitg@mat.uab.cat}\\ {\tt jvm@mat.uab.cat}\\*[5pt]

\end{tabular}}
\end{document}